\documentclass[a4paper, reqno]{amsart}
\usepackage{mathpazo, color}
\usepackage{enumerate}
\usepackage{graphics}
\usepackage{mathrsfs}

\newtheorem{theorem}{Theorem}
\newtheorem*{theorem*}{Theorem}
\newtheorem{lemma}[theorem]{Lemma}
\newtheorem{corollary}[theorem]{Corollary}

\theoremstyle{remark}

\newtheorem{remark}[theorem]{Remark}

\def\N{{\mathbb{N}}}
\def\Z{{\mathbb{Z}}}
\def\Rl{{\mathbb{R}}}
\def\Cx{{\mathbb{C}}}

\def\tr{{\mathrm{tr}\,}}

\def\sgn{{\mathrm{sgn}\,}}
\def\Re{{\mathrm{Re}\,}}
\def\Im{{\mathrm{Im}\,}}
\def\cS{{\mathcal{S}}}
\def\cA{{\mathfrak{A}}}
\def\supp{{\mathrm{supp}\,}}
\def\Pl{{\mathcal{P}}}
\def\const{{\mathrm{const}\,}}
\def\cC{{\mathfrak{C}}}

\begin{document}

\title{Fr\'echet differentiability of~$\cS^p$ norms }
%\author{Denis Potapov}
%\author{Fedor Sukochev}

\author{Denis Potapov}
\email{d.potapov@unsw.edu.au}
\thanks{Research is partially supported by ARC}
%\thanks{Corresponding author: d.potapov@unsw.edu.au}
\author{Fedor Sukochev}
\email{f.sukochev@unsw.edu.au}
\address{School of Mathematics \& Statistics, University of NSW,
  Kensington NSW 2052 AUSTRALIA}

\maketitle

\bibliographystyle{short}

\begin{abstract} One of the long standing questions in the theory 
%of non-commutative $L_p$-spaces 
of Schatten-von Neumann ideals of compact operators
is whether their norms have the same differentiability properties as the norms of their commutative counterparts. 
We answer this question in the affirmative.
% in the setting of Schatten-von Neumann ideals of compact operators. 
A key technical observation underlying our proof is a discovery of connection between this question and recent 
affirmative resolution of L.S. Koplienko's conjecture concerning existence of higher order spectral shift functions.
\end{abstract}

It was conjectured  in~\cite{TJ1975Diff} and ~\cite[Remark, p.35]{AF-1992}, that
the norm of the
Schatten-von Neumann class~$S^p$ on an arbitrary real Hilbert space ${\mathcal
H}$ is
$[p]$-times Fr\'echet differentiable\footnote{Symbol~$[\cdot]$ stands
  for the integral part function.} for any $1 < p < \infty$, $p \not\in\N$. 
That is,

\begin{theorem}\label{TJ}
  The function~$H \mapsto \left\| H \right\|_p^p$, $H \in \cS^p$, $1 <
  p < \infty$ is $m$-times Fr\'echet differentiable away from zero, where~$m = [p]$
  and~$p \not \in \N$.
\end{theorem}

This manuscript gives a proof to this conjecture.  Let ${\mathcal H}$ be an
arbitrary complex Hilbert space and let us consider the
Schatten-von Neumann class~$S^p$ associated with ${\mathcal H}$ as a Banach
space over the field of real numbers. We prove (in Theorem~\ref{TaylorExpansion} below)
the following Taylor expansion result (for all relevant definitions and
terminology concerning differentials of abstract functions
we refer to \cite{LS-1961}).

\def\dfun#1{\delta^{(#1)}_H}

\begin{theorem*}
  If~$H \in \cS^p$, $\|H\|_p\leq 1$, $1 < p < \infty$ and if~$m \in \N$ is such
  that~$m <p \leq m + 1$, then there are bounded symmetric polylinear forms
$\dfun{k}$, $1\leq k\leq m$
  \begin{equation*}
    \dfun{1} : \cS^p \mapsto \Rl,\ \
    \dfun{2} : \cS^p \times \cS^p \mapsto \Rl,\ \
    \ldots,\ \
    \dfun{m}: \underbrace{\cS^p \times \ldots \times
      \cS^p}_{\text{$m$-times}} \mapsto \Rl
  \end{equation*}
  such that
  \begin{equation}
    \label{eq:TaylorExpansion}
    \left\| H + V \right\|_p^p - \left\| H \right\|_p^p -
    \sum_{k = 1}^m \dfun{k} \Bigl(\underbrace{V, \ldots,
      V}_{\text{$k$-times}}\Bigr) = O(\left\| V \right\|_p^p),
  \end{equation}
  where~$V \in \cS^p$ and~$\left\| V \right\|_p \rightarrow 0$.
\end{theorem*}

Taking into account that for an even $p\in \N$ the norm of ~$S^p$ is obviously
infinitely many times differentiable, we indeed confirm the conjecture
that the norm of ~$S^p$ and that of their classical counterpart
$\ell^p$ share the same differentiability properties. For the proofs of
corresponding commutative results see \cite{BF-1966} and \cite{S-1967}.

Our techniques are based on a new approach to and results from the theory of multiple operator integration presented in \cite{PSS-SSF}. In that paper the authors 
(jointly with A. Skripka) applied that theory to resolve L.S. Koplienko's conjecture that a spectral shift function exists for every integral $p>2$. 
Suitably enhanced and strengthened technical estimates from \cite{PSS-SSF} and its companion paper \cite{PSS-SSF2} are crucially used in the proofs below. 

Finally, we mention that a closely related problem concerning differentiability properties of the norm of general non-commutative $L_p$-spaces 
was also stated in \cite{PX}. Our methods allow a further extension to cover also a case of $L_p$-spaces associated with an arbitrary  semifinite von Neumann algebra 
$\mathcal M$. 
This extension will be presented in a separate article, however in the last section we resolve this problem for a special case when von Neumann algebra 
$\mathcal M$ is of type $I$.

\section{Multiple operator integrals}
\label{sec:moi}

The proof of Theorem~\ref{TaylorExpansion} is based on methods drawn from the
theory of multiple operator integrals.  A brief account of that theory is given
below together with some new results.

\subsection*{Multiple operator integrals from \cite{Peller2006} and \cite{ACDS}}

Let~$\cA_m$ be the class of functions~$\phi: \Rl^{m + 1} \mapsto \Cx$
admitting the representation
\begin{equation}
  \label{Arep}
  \phi (x_0, \ldots, x_m) =
  \int_\Omega \prod_{j = 0}^m a_j (x_j, \omega)\, d\mu(\omega),
\end{equation}
for some finite measure space~$\left( \Omega, \mu \right)$ and
bounded Borel functions $$ a_{j} \left( \cdot, \omega \right) :
\Rl \mapsto \Cx. $$ The class~$\cA_m$ is in fact an algebra with respect to the operations of pointwise addition and multiplication 
\cite[Proposition 4.10]{ACDS}. The formula
$$ \left\| \phi
\right\|_{\cA_m} = \inf \int_\Omega \prod_{j = 0}^m \left\|
a_j(\cdot, \omega) \right\|_\infty \, d\left| \mu \right|(\omega),
$$ where the infimum is taken over all possible
representations~(\ref{Arep}) defines a norm on ~$\cA_m$ (see \cite{PS-DiffP}).

For every~$\phi \in \cA_m$, and a fixed $(m +
1)$-tuple of self-adjoint operators~$\tilde H :=
\left( H_0, \ldots, H_{m} \right)$, the multiple operator integral
\begin{equation}
  \label{MOIdef}
  T_\phi : \cS^{p_1} \times \ldots \times
  \cS^{p_m} \mapsto \cS^p,\ \ \text{where}\ \frac 1p = \frac 1{p_1} +
  \ldots + \frac 1{p_m}.
\end{equation}
is defined as follows
\begin{multline*}
  T_\phi \left(
    V_1, \ldots, V_m \right) = \int_\Omega a_0(H_0, \omega)\, V_1 \, a_1(H_1,
  \omega)\, \cdot \ldots \cdot V_m a_m(H_m, \omega)\, d\mu(\omega),\\ V_j \in
  \cS^{p_j},\ \ j = 1,\ldots,m.
\end{multline*}
Here~$a_j$'s and~$(\Omega, \mu)$ are taken from the representation~(\ref{Arep})
and one of the  main results of this theory is that the value $T_\phi \left(
    V_1, \ldots, V_m \right)$ does not depend on that representation \cite[Lemma
3.1]{Peller2006}, \cite[Lemma 4.3]{ACDS}.
 If it is necessary to specify the $(m+1)$-tuple $\tilde H$ used in the
definition of the multiple operator integral $T_\phi$,
we write $T_\phi^{\tilde H}$. Observe that~$T_\phi$ is a multilinear
operator and, if~$1 \leq p \leq \infty$, $T_\phi$ is bounded, i.e.,
\begin{equation}
  \label{TphiBasicEst}
  \left\| T_\phi \right\|_{\tilde{p}\to p} \leq \left\| \phi \right\|_{\cA_m},
\end{equation}
where the norm~$\left\| T_\phi \right\|_{\tilde{p}\to p}$ is the
norm of multilinear operator, that is $$ \left\| T_\phi
\right\|_{\tilde{p}\to p} := \sup \left\| T_\phi(V_1,
  \ldots, V_m) \right\|_p, $$ where $\tilde{p}=(p_1,\ldots ,p_m)$ and the supremum
is taken over all
$m$-tuples~$(V_1, \ldots, V_m)$ such that~$\left\| V_j
\right\|_{p_j} \leq 1$, $j = 1, \ldots, m$. The proof of this
assertion follows along the same line of thought as in \cite[Lemma
4.6]{ACDS} (see also \cite[Section 4.1]{AlPe-Sp-2010} and
\cite[Lemma 3.5]{PSS-SSF}).

The transformation~$T_\phi$ with~$\phi \in \cA_m$ defined above
has the following simple algebraic property \cite[Proposition 4.10(ii)]{ACDS}.  If~$b_j$, $j = 0,
\ldots, m$ are bounded Borel functions, then~$\psi \in \cA_m$,
where $$ \psi(x_0, \ldots, x_m) = b_0(x_0)\cdot \ldots \cdot b_n
(x_m)\, \phi(x_0, \ldots, x_m) $$ and
\begin{equation}
  \label{TphiHomomorphism}
  T_\psi(V_1, \ldots, V_m) =
  T_\phi(V_1', \ldots, V_m')
\end{equation}
where $$ V_1' = b_0(H_0)\, V_1 b_1(H_1) \ \ \text{and}\ \ V_j' =
V_j \, b_j(H_j),\ \ j = 2, \ldots, m. $$  In particular,
if~ $\psi(x_0, \ldots, x_m) = x_0^{s_0} \cdot \ldots \cdot x_m^{s_m}\, \phi(x_0,
\ldots, x_m)$, where $s_0,\dots, s_m$ are non-negative integers and $\tilde{H}$ consists of bounded operators,
then
\begin{equation}
  \label{TphiHomomorphismI}
  T_\psi(V_1, \ldots, V_m) = T_\phi(H_0^{s_0}V_1\, H_1^{s_1}, V_2 H_2^{s_2},
  \ldots, V_m \, H_m^{s_m}).
\end{equation}

\subsection*{A version of multiple operator
integrals from \cite{PSS-SSF}}

We shall also need a closely related but distinct version of operators $T_\phi$
introduced recently in \cite{PSS-SSF}.

Let $m\in\mathbb N$.  Let~$dE^j_\lambda$, $\lambda \in \Rl$ be the
spectral measure corresponding to the self-adjoint operator $H_j$ from
the $(m+1)$-tuple $\tilde H$.  We set~$E_{l,n}^j = E^j \left[ \frac
  l n, \frac {l + 1}n \right)$, for every~$n \in \N$ and~$l \in \Z$,
where $E^j[a,b)$ is the spectral projection of the operator $H_j$
corresponding to the semi-interval $[a,b)$.

Let $1\le p_j\le\infty$, with $1\le j\le m,$ be such that
$0\le\frac{1}{p_1}+\ldots +\frac{1}{p_m}\le 1.$ Let $V_j\in \mathcal
S^{p_j}$ and denote $\tilde V = \left( V_1, \ldots, V_{m} \right).$
Fix a bounded Borel function $\phi: \mathbb R^{m+1}\mapsto\mathbb C.$
Suppose that for every $n\in\mathbb N$ the series $$S_{\phi,n}(\tilde
V):=\sum\limits_{l_0,\ldots ,l_m\in\mathbb
  Z}\phi\left(\frac{l_0}{n},\ldots ,\frac{l_m}{n}\right)E^0_{l_0,n}V_1E^1_{l_1,n}
V_2\cdot\ldots \cdot V_nE^m_{l_m,n}$$ converges in the norm of $\mathcal
S^p,$ where $\frac1p=\frac1{p_1}+\ldots +\frac1{p_m}$ and $$\tilde
V\mapsto S_{\phi,n}(\tilde V),\quad n\in\mathbb N,$$ is a sequence of
bounded multilinear operators $\mathcal S^{p_1}\times\ldots \times\mathcal
S^{p_m}\mapsto \mathcal S^{p}.$ If the sequence of operators
$\{S_{\phi,n}\}_{n\ge 1}$ converges strongly to some multilinear
operator $\hat T_\phi,$ then, according to the Banach-Steinhaus
theorem, $\{S_{\phi,n}\}_{n\ge 1}$ is uniformly bounded and the
operator $\hat T_\phi$ is also bounded. In this case the operator
$\hat T_\phi$ is called a modified multiple operator integral.  If it
is necessary to specify the $(m+1)$-tuple $\tilde H$ used in the
definition of the multiple operator integral $\hat T_\phi$, we write
$\hat T_\phi^{\tilde H}$.

Let~$\cC_m$ be the class of functions~$\phi: \Rl^{m + 1} \mapsto \Cx$
admitting the representation \eqref{Arep} with bounded
continuous functions $$ a_{j} \left( \cdot, s \right) : \Rl \mapsto
\Cx $$ for which there is a growing sequence of measurable
subsets~$\left\{\Omega^{(k)} \right\}_{k \geq 1}$, with $\Omega^{(k)}  \subseteq
\Omega$ and~$\cup_{k \geq 1} \Omega^{(k)}  = \Omega$ such that the
families $$ \left\{a_{j} (\cdot, s) \right\}_{s \in \Omega^{(k)} },\ \ 0
\leq j \leq m, $$ are uniformly bounded and uniformly equicontinuous.  The
class~$\cC_m$ has the norm $$ \left\| \phi \right\|_{\cC_m} = \inf
\int_\Omega \prod_{j = 0}^m \left\| a_j(\cdot, s) \right\|_\infty \,
d\left| \mu \right|(s), $$ where the infimum is taken over all
possible representations~\eqref{Arep} as specified above. Hence, we have
\begin{equation}
  \label{AnlessthanCn}
 \left\| \phi\right\|_{\cA_m}\leq \left\| \phi \right\|_{\cC_m},\ \forall \phi
\in \cC_m.
\end{equation}

The following lemma demonstrates a connection between two types of
operator integrals $\hat T_\phi$ and $T_\phi$.

\begin{lemma}[{\cite[Lemma 3.5]{PSS-SSF}}]\label{lemma1}
Let $1\le p_j\le \infty,$ with $1\le j\le m,$ be such that
$0\le\frac{1}{p_1}+\ldots +\frac{1}{p_m}\le 1.$ For every $\phi
\in\cC_m$, the operator $\hat T_\phi$ exists and is bounded on
$\mathcal S^{p_1}\times\ldots \times\mathcal S^{p_m},$ with
\begin{equation}
  \label{hatTphiBasicEst}
\|\hat T_\phi\|_{\tilde p\to p}\le\|\phi\|_{\cC_m}.
\end{equation}
Moreover,  $\hat T_\phi=T_\phi$.
\end{lemma}
The result above is stated in \cite{PSS-SSF} under the additional assumption
that
$$\tilde H=(H,H,\dots, H),$$
however, it is straightforward to see that the latter restriction is redundant.

It is important to observe that the class of functions to which the definition
from \cite{Peller2006} and \cite{ACDS} is applicable is distinct
from the class of functions for which the definition from \cite{PSS-SSF} makes
sense. 
Observe also that the algebraic relations~(\ref{TphiHomomorphism})
and~(\ref{TphiHomomorphismI}) continue to hold for the modified
operators~$\hat T_\phi$ (see~\cite[Lemma~3.2]{PSS-SSF}).

\subsection*{Besov spaces}

For the function $f\in L_1$ by $\hat{f}$ we denote
its Fourier transform, i. e.,
$$
\hat{f}(t)=\int\limits_{\mathbb R}f(x)e^{-ixt}dx.
$$
We shall also sometimes use the same symbol for Fourier transform of a tempered distribution.

For a given $s\in \Rl$, the homogeneous Besov space~$\dot B_{\infty 1}^s$ is the
collection of all generalized functions on~$\Rl$ satisfying the
inequality $$ \left\| f \right\|_{\dot B^s_{\infty 1}} := \sum_{n
\in \Z} 2^{sn} \left\| f * W_n \right\|_\infty < + \infty, $$
where
\begin{equation}
  \label{WnFuncs}
  W_n (x) = 2^n
  W_0 (2^n x),\ \ x \in \Rl,\ \ n \in \Z
\end{equation}
and~$W_0$ is a smooth
function whose Fourier transform is like on fig.~\ref{fig:w0}.  We
also require that\footnote{This condition ensures that $$ \sum_{n \in
    \Z} \hat W_n (x) = 1,\ \ x \neq 0. $$} $$ \hat W_0 (y) + \hat W_0
(\frac y2) = 1,\ \ 1 \leq y \leq 2. $$

\begin{figure}[h]
  \centering
  \includegraphics{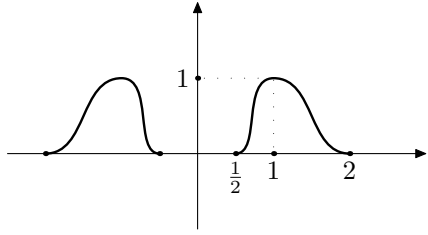}
  \caption{The Fourier transform~$\hat W_0$}
  \label{fig:w0}
\end{figure}

Observe that if~$f$ is a polynomial, then its Fourier transform is
supported at~$x = 0$, and since the functions $W_n$ are not supported at $0$, we have~$\left\| f \right\|_{\dot B^s_{\infty 1}} =
0$.  The modified homogeneous Besov class~$\tilde B^s_{\infty 1}$ is
given by $$ \tilde B^s_{\infty 1} = \left\{f \in \dot B^s_{\infty 1},\
  \ f^{([s])} \in L^\infty:=L_\infty(\Rl) \right\}, $$ where~$[s]$ is the integral
part of~$s$.  The value of the norm $\left\| \cdot\right\|_{\dot
B^s_{\infty 1}}$ on the elements of $\tilde B^s_{\infty 1}$ will
 be denoted by $\left\| \cdot\right\|_{\tilde B^s_{\infty
1}}$.  Note also that the norms $\|f\|_{\tilde
B^s_{\infty 1}}$ and $\|f^{([s])}\|_{\tilde B^{s-[s]}_{\infty 1}}$ are
equivalent on the space $\tilde B^s_{\infty 1}.$ Indeed, this equivalence may be
easily inferred from
\cite[(36)]{Triebel-1982} (all what one needs to recall is that the Poisson
integral used in \cite[(36)]{Triebel-1982} commutes with the
 differentiation, that is $P(t)f^{(k)} = ( P(t)f )^{(k)}$).

The elements of~$\tilde
B^s_{\infty 1}$ can also be described as follows
\begin{multline*}
  f \in \tilde B^s_{\infty 1}\ \ \Longleftrightarrow\ \ f(x) = c_0 +
  c_1 x + \ldots + c_m x^m + f_0(x),\\ c_j \in \Cx,\ j = 0,\ldots,
  m,\ f_0 \in \dot B^s_{\infty 1},\ \supp \hat f_0 \subseteq \Rl
  \setminus \{0\}.
\end{multline*}

Recall that~$\Lambda_\alpha$ is the class of all H\"older
functions of exponent~$0 < \alpha < 1$, that is the functions
$f:\mathbb R\to \mathbb C$ such that
$$\|f\|_{\Lambda_\alpha}:=\sup\limits_{t_1,t_2}\frac{|f(t_1)-f(t_2)|}{
|t_1-t_2|^\alpha}<+\infty.$$

We also need the following simple criterion.
\begin{lemma}
  \label{HolderToBesov}
  If\/~$f^{(m-1)} \in \Lambda_{1-\epsilon}$ and\/~$f^{(m)} \in
  \Lambda_\epsilon$, for\/~$0 < \epsilon < 1$ for some~$m \in \N$,
  then~$f \in \tilde B^m_{\infty 1}$.
\end{lemma}

The case~$m = 1$ is proved in~\cite[proof of Theorem
4]{PoSuNGapps}, the proof of the general case is identical to the
case~$m = 1$. We leave details to the reader.

\subsection*{Polynomial integral momenta.}
\label{sec:IPM}

Let $\Pl_m$ be the class of polynomials of~$m$ variables with real
coefficients.  Let~$S_m$ be the simplex $$ S_m := \left\{ \left( s_0,
    \ldots, s_m \right) \in \Rl^{m + 1}:\ \ \sum_{j = 0}^m s_j = 1,\ \
  s_j \geq 0,\ \ j = 0, \ldots, m \right\} $$ and let
\begin{equation*}
  \label{Rn}
  R_m := \left\{(s_1, \ldots, s_{m})
    \in \Rl^m:\ \ \sum_{j = 1}^{m} s_j \leq 1,\ \ s_j \geq 0,\ \ j=1,\ldots, m\right\}.
\end{equation*}
We equip the
simplex~$S_m$ with the finite measure~$d\sigma_m$ defined by the requirement
that the equality
\begin{equation}
\label{minv} \int_{S_m} \phi(s_0, \ldots, s_m)\, d\sigma_m =
\int_{R_m} \phi\left( 1 - \sum_{j =
    1}^{m}s_j,\, s_1, \ldots, s_{m}  \right)\, dv_m,
\end{equation}
holds for every continuous function~$\phi: \Rl^{m+1} \mapsto \Cx$, where
~$dv_m$ is the Lebesgue measure on~$\Rl^m$.  It can be seen via a
straightforward change of variables in~(\ref{minv}) that the
measure~$d\sigma_m$ is invariant under any permutation of the
variables~$s_0, \ldots, s_m$.

Let $ \tilde s = \left( s_1, \ldots, s_m \right) \in R_m$ and let $(s_0,
\tilde s) \in S_m$, that is $ s_0=1-\sum_{j=1}^m s_j.$ Given $h \in
L^\infty$, and $Q \in \Pl_m$, we set
\begin{equation}
\label{2stars}
\phi_{m, h, Q} \left( \tilde x \right) = \int_{S_m} Q\left(
  \tilde s \right)\, h \left( \sum_{j = 0}^m s_j x_j \right)
\, d\sigma_m,
\end{equation}
where $\tilde x = \left( x_0, \ldots, x_m \right) \in \Rl^{m+1}$.
Following the terminology set in \cite{PSS-SSF} we shall call the
function~$\phi_{m, h, Q}$ a {\it
  polynomial integral momentum}. This notion plays a crucial role in our present
approach. Indeed, the functions
~$\phi: \Rl^{m+1} \mapsto \Cx$ for which we shall be considering multiple
operator integrals $T_\phi$ and $\hat T_\phi$ are in fact
of the form ~$\phi_{m, h, Q}$ for suitable choice of $h$ and $Q$.

\subsection*{Multiple operator integral of a polynomial integral
  momentum }
\label{moi-of-pim}
In this subsection, we describe the connection between the norm $
    \left\| \phi_{m, h, Q} \right\|_{\cC_m}$ and a norm of the function $h$ and
thereby connect the latter with the norm
$ \left\| T_\phi \right\|_{\tilde{p}\to p}$ (see \eqref{AnlessthanCn} and \eqref{TphiBasicEst}).
The following result extends~\cite[Theorem~5.1]{Peller2006}.  It also
improves~\cite[Lemma~5.2]{PSS-SSF}.

\begin{theorem}
  \label{PIMpolyEst}
  Let~$\phi = \phi_{m, h, Q}$ be a polynomial integral momentum.
  \begin{enumerate}
  \item If\/~$\supp \hat h \subseteq [2^{N-1}, { 2^{N+1}}]$, for some~$N \in
    \Z$, then~$\phi \in \cC_m$ and $$ \left\| \phi \right\|_{\cC_m}
    \leq \const \left\| h \right\|_\infty. $$\label{PIMpolyEst:Exponential}
  \item If\/~$h \in \tilde B^0_{\infty 1}$, then~$\phi \in \cC_m$ and $$
    \left\| \phi \right\|_{\cC_m} \leq \const \left\|
      h\right\|_{\tilde B^0_{\infty 1}}. $$\label{PIMpolyEst:Besov}
  \end{enumerate}
The constants above do not depend on $N$ or $h$.
\end{theorem}

The proof is based on the following lemma

\begin{lemma}
  \label{PIMpolyLemma}
  \begin{enumerate}
  \item A tempered distribution~$r_a$, $a > 0$ defined via Fourier
    transform by $$ \hat r_a (y) = 1 - \frac a{\left| y \right|},\ \
    \text{if\/~$\left| y \right| > a$} \ \ \text{and}\ \ \hat r_a(y) =
    0,\ \ \text{otherwise} $$ is a finite measure whose total variation
satisfies $$ c_0 := \sup_{a >
      0} \left\| r_a \right\|_1 < + \infty. $$ \label{PIMpolyLemma:ra}
  \item In particular, if~$h \in L^\infty$ such that~$\supp \hat h
    \subseteq \Rl_+$, then $$ \left\| h_{a, { \gamma}} \right\|_\infty \leq c_0\,
    \left\| h \right\|_\infty,\ \ \forall a > 0,\ { \gamma
      \geq 1}, $$ where $$ h_{a, { \gamma}} (x) :=
    \int_0^\infty { \left[ \frac y{y + a} \right]^\gamma} \, \hat h(y + a)\, e^{ixy}\,
    dy. $$ \label{PIMpolyLemma:ha}
  \item If~$\supp \hat h \subseteq [2^{N-1}, { 2^{N+1}}]$, then $$ \left\|
      {h_{m,N}} \right\|_\infty \leq \const\, 2^{mN}\, \left\| h\right\|_\infty,
$$
    where $$ {h_{m,N}} (x) = \int_0^\infty y^m \hat h(y)\, e^{ixy}\,
    dy. $$ \label{PIMpolyLemma:hm}
The constant above does not depend on $N$ and $h$.
  \end{enumerate}
\end{lemma}

\begin{proof}[Proof of Lemma~\ref{PIMpolyLemma}]
Combining \cite[Lemma 7]{PoSuNGapps} and the assumptions $$ 1 - \hat r_1 \in L^2(\Rl) \
  \ \text{and}\ \ \frac d{dy} \left( 1 - \hat r_1 \right) \in L^2(\Rl), $$
we see that the function~$1 - \hat r_1$ is a Fourier transform of an
  $L^1$-function.  Thus, $r_1$ is a finite measure  as a combination of
  former $L^1$-function and the Dirac delta function.

Observe further that $$ r_a (x) = a r_1 (a x),\ \ x \in \Rl, $$
so
$$  \left\| r_a \right\|_1 = \left\| r_1 \right\|_1. $$
This completes the proof of \eqref{PIMpolyLemma:ra}.

The part~(\ref{PIMpolyLemma:ha}) follows from  Young's  inequality and the observation
  that \begin{equation}\label{temporary1} h_{a, {
        \gamma}}(x) = \int_a^\infty { \left[ \frac {y - a}y \right]^\gamma}
\hat h(y)\, e^{ixy}
  e^{-ixa}\, dy = e^{-ixa}\,( \underbrace{ r_a * \ldots *
      r_a}_{\text{$\gamma$-times}} * h (x)). \end{equation}

  For the part~(\ref{PIMpolyLemma:hm}), we consider the
  function~$\delta_{m, 0}$ such that its Fourier transform is smooth
  and is as follows $$ \supp \hat \delta_{m, 0} \subseteq \left[ \frac
    14, { 4} \right] \ \ \text{and}\ \ \hat \delta_{m, 0} (y) = y^m,\ \
  \text{if}\ \ \frac 12 \leq y \leq { 2}. $$ By, e.g.,~\cite[Lemma 7]{PoSuNGapps},
  $\delta_{m, 0} \in L^1(\Rl)$.  We also set $$ \delta_{m, N} (x) =
  2^{(m + 1) N}\, \delta_{m, 0} \left( x 2^N \right). $$ Clearly,
  $\delta_{m, N} \in L^1(\Rl)$ and $$ \left\| \delta_{m, N} \right\|_1
  = 2^{mN} \left\| \delta_{m, 0} \right\|_1. $$ Observe also that on
  the Fourier side $$ \hat \delta_{m, N} (y) = 2^{mN}\, \hat
  \delta_{m, 0}(2^{-N} y). $$ In particular, $$ \hat \delta_{m, N} (y)
  = y^m,\ \ \text{if}\ \ 2^{N-1} \leq y \leq { 2^{N+1}}. $$
  The claim now
  follows from \begin{equation}\label{temporary2} {h_{m,N}} (x) = \delta_{m, N}
* h (x) \end{equation}
% {\color{red} Since $\supp \hat h\subset[2^{N-1},2^N]$ and $\hat
% \delta_{m, N} (y)$ coincides with $y_m$ on the support of $\hat
% h,$ we have
% $$\begin{array}{rl}h_m(x)&=\int_0^\infty y^m \hat h(y)\,
% e^{ixy}\,dy=\int_0^\infty \hat\delta_{m,N}(y) \hat h(y)\,
% e^{ixy}\,dy\\ &=\int_{0}^{\infty} \hat\delta_{m,N}(y)
% \big(\int_0^\infty h(s)e^{-isy}ds\big)\, e^{ixy}\,dy=\int_0^\infty
% h(s)\big(\int_{0}^{\infty} \hat\delta_{m,N}(y) e^{iy(x-s)} dy\big) ds\\
% &=\int_0^\infty h(s) \delta_{m,N}(x-s)ds=\delta_{m,N}* h(x).
% \end{array}$$
%
% or  something like that $$F(\delta_{m,N}*
% h(x))=F(\delta_{m,N}(x))F(h(x))=\hat\delta_{m,N}(y)\cdot \hat
% h(y)=y_m \cdot \hat h(y)=F(h_m(x)).$$ Hence, $h_m(x)=\delta_{m,N}*
% h(x).$}
 and Young's inequality.
\end{proof}

\begin{proof}[Proof of Theorem~\ref{PIMpolyEst}]
To prove part~(\ref{PIMpolyEst:Exponential}), we fix
a function ~$h$ such that~$\supp \hat h \subseteq [2^{N-1},
{ 2^{N+1}}]$. Using
  the definition of the integral momentum~$\phi_{m, h, Q}$ and the
  Fourier expansion $$ h(x) = \int_0^\infty \hat h(y)\, e^{ixy}\,
  dy, $$ we obtain
  \begin{multline*}
    \phi_{m, h, Q} (\tilde x) = \int_{S_m} Q(\tilde s)\, h(s_0 x_0 +
    \ldots + s_m x_m)\, d\sigma_m \\ = \int_0^\infty dy \int_{S_m}
    Q(\tilde s)\, \hat h(y)\, e^{iy s_0 x_0} \cdot \ldots \cdot
    e^{iy s_m x_m}\, d\sigma_m.
  \end{multline*}
We shall now make a substitution in the latter integration  via replacing the
current integration variables~$y$ and~$s_j$, $j = 1, \ldots, m$   
with the variables~$y_j$, $j = 0,
  \ldots, m$ such that $$ y_j = y s_j,\ \ \text{if~$j = 1, \ldots, m$}
  \ \ \text{and}\ \ y_0 = y s_0 = y \cdot \left( 1 - \sum_{j = 1}^m
    s_j \right). $$ This substitution transforms the domain of integration $$ y
\geq 0 \ \
  \text{and}\ \ s_j \geq 0 \ \ \text{and}\ \ s_1 + \ldots + s_m \leq
  1 $$ into the first octant $$ y_j \geq 0,\ \ j = 0,
  \ldots, m. $$ Introducing the notation
 \begin{equation*}
    J_{ k}(a) :=
    \begin{bmatrix}
      a & - y & -y & \cdots & -y & -y \\
      s_1 & y & 0 & \cdots & 0 & 0 \\
      s_2 & 0 & y & \cdots & 0 & 0 \\
      \vdots & \vdots & \vdots & \ddots & \vdots & \vdots \\
      s_{ k-1} & 0 & 0 & \cdots & y & 0 \\
      s_{ k} & 0 & 0 & \cdots & 0 & y
    \end{bmatrix} ,\quad k\le m,\,\,a\in\mathbb R,
  \end{equation*}
we observe that the Jacobian of our substitution is given by $$ J_m(1 -
  s_1 - \ldots - s_m).$$ 
Using the last column decomposition of the latter determinant, we
obtain $$
    J_k(a) = (-1)^{k} (-y) (-1)^{k-1} s_k y^{k-1}+
(-1)^{2k}y J_{k-1}(a)=(-1)^{2k}s_k
    y^{k}+y J_{k-1}(a)=$$$$s_k
    y^{k}+y J_{k-1}(a),\,\,\text{ for every }\,k\le m, \,\,a\in\mathbb R.$$
Thus,
  $$
  \begin{array}{ll}
    J_ k(a) & = y J_{ k-1} (a) + s_k y^{k} \\ &=
    y \left( y J_{ k-2}{(a)} + s_{{k}-1} y^{{ k}-1} \right) + s_{ k} y^{ k} \\
&=
      y^2 J_{{ k}-2}{(a)} + y^{ k}(s_{ k-1} + s_{ k})
    \ldots\\ &
     = y^{ k} J_0(a) + y^{ k} (s_1 + \ldots + s_{ k}),\quad {\text{ for every
}\, k\le m}.
  \end{array}
  $$
Since $J_0(1
  - s_1 - \ldots - s_m)=1
  - s_1 - \ldots - s_m$ , the Jacobian of the substitution above is $$J_m(1
  - s_1 - \ldots - s_m) = y^m. $$ Therefore, the integral
  momentum~$\phi_{m, h, Q}$ takes the form $$ \phi_{m, h, Q} (\tilde
  x) = \int_{\Rl_+^{m + 1}} Q(\tilde s)\, y^{-m}\, \hat h(y)\, e^{iy_0
    x_0} \cdot \ldots \cdot e^{iy_m x_m}\, dy_0\, \ldots dy_m, $$
  where $$ y: = y_0 + \ldots + y_m \ \ \text{and}\ \ \tilde s = (s_1,
  \ldots, s_m),\ \ s_j := \frac {y_j}{y},\ \ j = 1, \ldots, m. $$
Observe that, since~$\supp \hat h \subseteq [2^{N-1}, { 2^{N + 1}}]$, the  integration over~$\Rl_+^{m + 1}$ is in fact only taken over the strip 
$$ y_j \geq 0,\ \ j = 0, \ldots, m \ \ \text{and}\ \ 2^{N-1}
  \leq y_0 + \ldots + y_m \leq 2^{N+1}. $$
For the rest of the proof observe that it suffices to show the claim of the theorem for a  monomial $$ Q(\tilde s) = s_0^{\gamma_0} \cdot s_1^{\gamma_1} \cdot
  \ldots \cdot s_m^{\gamma_m}. $$ For such monomial, we shall consider
  two different scenarios.

Assume first that the monomial~$Q \equiv 1$, i.e., $\gamma_j = 0$,
  $j = 0, \ldots, m$.  In this case, using the fact that $$ 1 = s_0 +
  \ldots + s_m, $$ we can split the monomial~$Q \equiv 1$ into $m+1$
  sum of monomials where not every~$\gamma_j$ vanishes,  i. e.  $Q(\tilde
s)=s_0+s_1+\ldots +s_m$.

So we have arrived at the second scenario.  Assume now that not all of
$\gamma_j$, $j = 0, \ldots, m$ vanish.  For simplicity,
  assume~$\gamma_0 \geq 1$.  In this case, we shall show that $\phi_{m,h,Q}$
admits a representation~(\ref{Arep}) with
$$ \Omega_N = \left\{ \tilde y = (y_1, \ldots, y_m):\ \ y_j \geq 0,\ j = 1,
\ldots, m \ \ \text{and}\ \  y_1 + \ldots + y_m \leq { 2^{N + 1}} \right\},$$
equipped with the (scalar multiple of) Lebesgue measure
$d\mu_N=\frac{\|h\|_\infty}{2^{mN}}d\mu$ on  $\mathbb R^m$ and 

\begin{multline*}
   a_j(x, \tilde y) = \frac {y_j^{\gamma_j}}{y^{\gamma_j}} e^{ix
      y_j},\ \ j = 1, \ldots, m \ \ \text{and}\\ a_0 (x, \tilde y) =
\frac{2^{mN}}{\|h\|_\infty}
    \int_0^\infty \frac {y^{\gamma_0}_0}{y^{\gamma_0}}\, y^{-m}\, \hat
    h(y_0 + y_1 + \ldots + y_m)\, e^{ixy_0}\, dy_0, \\ \text{where}\ \
  y: = y_0 + \ldots + y_m,\   \tilde y = (y_1, \ldots, y_m)\in \Omega_N.
\end{multline*}

It is obvious that
$$\|a_j(\cdot,\tilde{y})\|_{\infty}=(\frac{y_j}{y})^{\gamma_j}\leq 1,\quad
\|a_j'(\cdot,\tilde{y})\|_{\infty}=y_j(\frac{y_j}{y})^{\gamma_j}\leq 2^{N+1}$$
for all $j=1,\ldots,m.$ Hence, the functions $a_j(\cdot,\tilde{y}),$
$j=1,\ldots,m$,  $\tilde y \in \Omega_N$ are uniformly bounded and uniformly equicontinuous. 
We claim that the same conclusion also holds for the functions
$a_0(\cdot,\tilde{y})$, $\tilde y \in \Omega_N$. 

Firstly, we check that
$$ \left\|  a_0(\cdot, \tilde y ) \right\|_\infty \leq \const. $$
Indeed, using the notation of Lemma~\ref{PIMpolyLemma}, 
we see that  $$ a_0\left( x, \tilde y \right) = \frac {2^{mN}}{\left\| h
  \right\|_\infty}\, \tilde h_{a, \gamma_0} \left( x \right),\ \
\text{where}\ \ a = y_1 + \ldots + y_m,\ \tilde h(x) = h_{-m,
  N}(x). $$ Thus, by Lemma~\ref{PIMpolyLemma}(\ref{PIMpolyLemma:ha})
and (\ref{PIMpolyLemma:hm}) we have $$ \left\| a_0 (\cdot, \tilde
  y)\right\|_\infty \leq \frac {2^{mN}}{\left\| h \right\|_\infty}
\left\| \tilde h_{a, \gamma_0} \right\|_\infty \leq \frac
{2^{mN}}{\left\| h \right\|_\infty} \left\|  h_{-m, N} \right\|_\infty
\leq \const. $$

Secondly, we claim that  the derivative $\frac{d}{dx}a_0 (x, \tilde y)$  is a
uniformly bounded function.
Indeed, writing this derivative as 
$$
\frac{d}{dx}a_0 (x, \tilde y) = \frac{2^{mN}}{\|h\|_\infty}
    \int_0^\infty \frac {y^{\gamma_0+1}_0}{y^{\gamma_0+1}}\, y^{-m+1}\, \hat
    h(y_0 + y_1 + \ldots + y_m)\, e^{ixy_0}\, dy_0, 
$$
and repeating  the argument used above with~$(\gamma_0+1)$  proves the uniform
boundedness of 
$a_0 (\cdot, \tilde y)$.

Now observing that
\begin{multline*}
  \int_{\Omega_N}\prod_{j=0}^ma_j(x_j,
\omega)d\mu_N \\ = \int_{\Omega_N} \left[ \int_0^\infty \frac
{y^{\gamma_0}_0}{y^{\gamma_0}}\, y^{-m}\, \hat
    h(y)\, e^{ix_0y_0}\, dy_0\, \right]\prod_{j=1}^m \frac
    {y^{\gamma_j}_j}{y^{\gamma_j}}e^{ix_jy_j}dy_1\cdot \ldots \cdot dy_m \\
=\int_{\mathbb R^{m+1}_+}  \prod_{j=0}^m \frac
{y^{\gamma_j}_j}{y^{\gamma_j}}y^{-m}\, \hat
    h(y)\, e^{ix_0y_0}\,
e^{ix_1y_1}\,\cdot \ldots \cdot \,e^{ix_my_m}\,dy_0\,dy_1\cdot \ldots \cdot dy_m\\
=\int_{\mathbb
R^{m+1}_+}  \prod_{j=0}^m s_j^{\gamma_j}y^{-m}\, \hat
    h(y)\, e^{ix_0y_0}\, e^{ix_1y_1}\,\cdot \ldots \cdot
    \,e^{ix_my_m}\,dy_0\,dy_1\cdot \ldots \cdot dy_m\\
    =\int_{\mathbb R^{m+1}_+} Q(\tilde s)y^{-m}\, \hat
    h(y)\, e^{ix_0y_0}\, e^{ix_1y_1}\,\cdot \ldots \cdot
    \,e^{ix_my_m}\,dy_0\,dy_1\cdot \ldots \cdot dy_m\\
    =\phi_{m, h, Q}(x_0,x_1,\cdot \ldots \cdot ,x_m),
\end{multline*}
we see that~$\phi_{m, h, Q}$ admits a representation
(\ref{Arep}) with are uniformly bounded and uniformly equicontinuous functions
$\{a_j(\cdot, \omega)\}_{\omega\in \Omega_N},\ 0\leq j\leq m$ (in this case, we have $\Omega^{(k)}=\Omega_N$ for all $k\ge 1$).

The proof of part~(\ref{PIMpolyEst:Besov}) is based on the estimates
obtained in the proof of part ~(\ref{PIMpolyEst:Exponential}) and the
approach from the proof of~\cite[Theorem~5.1]{Peller2006}. Observe that every element $h\in \tilde B^0_{\infty, 1}$ may be represented
as a uniformly convergent infinite sum $h = \sum_{N\in \mathbb Z}
h_N$,  where~$h_N = h * W_N$, of uniformly bounded functions $h_N$ such
that $\supp \hat h_N$ is contained in $[2^{N-1}, 2^{N+1}]$, for every
$N\in \mathbb{Z}$. Now, let $\left( \Omega, \mu \right) $ be the
direct sum of the measure spaces $\left( \Omega_N,\mu_N\right) $,
$N\in \mathbb{Z}$ (so $\Omega=\cdots \sqcup\Omega_{1}\sqcup
\Omega_{2}\sqcup\cdots$ is given by the disjoint union of
$\Omega_N$'s). Recalling from the proof above that $\mu_N(\Omega_N)=
\frac{\|h\|_\infty}{2^{mN}} \mu (\Omega_N)=\const \|h_N\|_\infty$
(here the constant does not depend on $N$), we see that the assumption
$\sum _{N\in \mathbb{Z}}\|h_N\|_\infty<\infty$ guarantees that $\left(
  \Omega, \mu \right) $ is a measurable space with finite
($\sigma$-additive) measure.

The definition of the functions $\{a_j(\cdot, \omega)\}_{\omega\in \Omega},\
0\leq j\leq m$ is now straightforward: 
the value of any such function for $\omega \in \Omega_N$ is given by the value
of the corresponding function defined in the 
proof of part ~(\ref{PIMpolyEst:Exponential}). It remains to verify that there
exists a growing sequence $\{\Omega^{(k)}\}_{k\ge 1}$ of measurable subsets of 
$\Omega$ such that for every $\omega \in \Omega$ there exists $k$ so that
$\omega\in \Omega^{(k)}$ and such that 
the families $\{a_j(\cdot, \omega)\}_{\omega\in \Omega^{(k)}},\ 0\leq j\leq m$
consist of uniformly bounded and uniformly equicontinuous functions. 
To this end, it is sufficient to set
$\Omega^{(k)}:=\sqcup_{{ N\leq k}}\Omega_{N}$ and refer to the results from part ~(\ref{PIMpolyEst:Exponential}).
This completes the proof of the theorem.
\end{proof}

\subsection*{A modified multiple operator integral of a polynomial integral
  momentum}

We continue discussing the polynomial integral momentum ~$\phi_{m, h, Q}$  of
order~$m$
associated with the function~$h \in L^\infty$ and a polynomial~$Q \in\Pl_m$.

If~$h \in \tilde   B^0_{\infty 1}$, then it follows from Theorem
~\ref{PIMpolyEst} and \eqref{AnlessthanCn} 
%see also Lemma ~\ref{lemma1} 
that $T_\phi$ is well defined.
However, in the case when~$h \not \in \tilde B^0_{\infty 1}$,  it is not
generally true that~$\phi = \phi_{m, h, Q} \in \cA_m$ and therefore the
definition~(\ref{MOIdef}) of the operator~$T_\phi$ associated
with~$\phi$ no longer makes any sense. In this latter case,
we have to resort to the modified operator integral $\hat T_\phi$.  

An important result established in~\cite{PSS-SSF}, which we shall exploit here
is that the concept of multiple
operator integral ~$\hat T_{\phi_{m, h, Q}}$ can be successfully defined under
the assumptions that~$h \in
C_b$ (continuous and bounded) and that the index $p$ in
~(\ref{MOIdef}) satisfies ~$1 < p < \infty$. The former
assumption is rather auxiliary and can likely be further relaxed,
whereas the later is principal. 

\begin{theorem}[{\cite[Theorem~5.3]{PSS-SSF}}]
  \label{PSStheorem}
  Let~$h \in C_b,$ $m\ge 1,$ $Q\in \mathcal P_m$. Let ~$\hat T_\phi=\hat
T_\phi^{\tilde H}$ be the modified multiple operator integral
  associated with a polynomial integral momentum $\phi = \phi_{m, h, Q}$ and an
arbitrary $(m+1)$-tuple of bounded self-adjoint operators
$\tilde H=(H_0,\ldots ,H_m)$.  If~$1 < p < \infty$ and $\tilde p=(p_1,\dots, p_m)$, $1<p_j<\infty$, $1\leq j\leq m$ satisfies the equality 
$\frac{1}{p}=\frac{1}{p_1}+\cdots+\frac{1}{p_m}$, then
\begin{equation}\label{PSStheoremEst} 
    \left\|   \hat T_\phi \right\|_{\tilde p\to p} \leq \const \left\| h 
\right\|_\infty.
\end{equation}
\end{theorem}
We briefly explain why the estimate ~(\ref{PSStheoremEst}) is far more superior
than any previously available estimates (in particular ~(\ref{TphiBasicEst}) and ~(\ref{hatTphiBasicEst})) 
of the norm of a multiple operator integral. For example, for
special integral polynomial momenta given by divided differences
(we explain this notion below in some detail), the best earlier available
estimate follows from a combination of
\eqref{hatTphiBasicEst} and \cite[Lemma 2.3]{ACDS} yielding

$$\left\| \hat T_\phi \right\|_{\tilde p\to p}\leq \const  \int_{\mathbb{R}}
\left|
  \widehat{h^{(m)}} (s) \right|\, ds.$$

Of course, the condition that the function $h$ (or its derivatives) has an
absolutely integrable Fourier transform is very restrictive.
Even in the case when~$h \in \tilde B^0_{\infty 1}$ and when we deal with the
'classical' multiple operator integral~$T_\phi$
(defined via Theorem~\ref{PIMpolyEst}), the best estimate available from a
combination of \eqref{TphiBasicEst} and Theorem \ref{PIMpolyEst}(ii)
$$\left\| T_\phi \right\|_{\tilde p\to p}\leq\const \left\|   h\right\|_{\tilde
B^0_{\infty 1}},$$
is still much weaker than the estimate \eqref{PSStheoremEst}. To
see that Theorem \ref{PSStheorem} is applicable here, observe that
the assumption ~$h \in \tilde B^0_{\infty 1}$ guarantees that the
corresponding integral momentum~$\phi\in \cC_m$ (the latter
assertion is proved in Theorem
\ref{PIMpolyEst}) and hence, by Lemma \ref{lemma1}, we may replace $\hat
T_\phi$ on the right hand side of \eqref{PSStheoremEst} with the
operator $T_\phi$.

In the special case when $m=1$, the result of Theorem \ref{PSStheorem} may be
found in \cite{PS-Lipschitz}. For an arbitrary $m\in \N$, this result was proved in
~\cite{PSS-SSF} under an additional assumption that the
$(m+1)$-tuple $\tilde H$ consists of identical operators. The proof of  Theorem
\ref{PSStheorem} follows by a careful
inspection of the proof of ~\cite[Theorem~5.3]{PSS-SSF}, which shows that the
argument there continues to stand if this additional assumption is omitted.
We leave further details to the reader.

We shall need a small  addendum to Theorem~\ref{PSStheorem}, which may be viewed as a variant of 
(Weak) Dominated Convergence Lemma for modified operator integrals of polynomial integral momenta.

\begin{lemma}
  \label{DomConvLemma}
  Let $h_n, h \in C_b$ be compactly supported functions such that $$ \lim_{n \rightarrow \infty} h_n(x)
  = h(x),\ \ \forall x \in \Rl. $$ Let also~$\phi_n = \phi_{m, h_n,
    Q}$, and~$\phi = \phi_{m, h, Q}$ be the polynomial integral
  momenta associated with~$Q \in P_m$ and the functions~$h_n$
  and~$h$ respectively. If\/~$\sup_{n} \left\| h_n \right\|_\infty < + \infty$,
  then the sequence of operators~$\left\{T_{\phi_n}\right\}$ converges
  to~$T_\phi$ weakly, i.e., $$ \lim_{n \rightarrow \infty} \tr \left(
    V_0\, T_{\phi_n} \left( \tilde V \right) \right) = \tr \left(
    V_0\, T_{\phi} \left( \tilde V \right) \right),\ \ \tilde V =
  \left( V_1, \ldots, V_m \right), $$ for every~$V_j \in \cS^{p_j}$,
  where~$1 < p_j < \infty$ for every~$j = 0, \ldots, m$ and~$\sum_{j =
    0}^m \frac 1 p_j = 1$.  In particular, $$ \left\| T_\phi
  \right\|_{\tilde p \rightarrow p_0'} \leq \liminf_{n \rightarrow
    \infty} \left\| T_{\phi_n} \right\|_{\tilde p \rightarrow p_0'},\
  \ \text{where}\ \tilde p = \left( p_1, \ldots, p_m \right),\
  \ \text{and}\ \frac{1}{p_0}+\frac{1}{p_0'}=1. $$
\end{lemma}

\begin{proof}[Proof of Lemma~\ref{DomConvLemma}]
  Fix~$V_j$ as in the statement of the lemma.  According to
  Theorem~\ref{PSStheorem}, the mapping $$ h \in C_b \mapsto \psi(h): =
  \tr \left( V_0\, T_{\phi} \left( \tilde V \right)
  \right) $$ is a continuous linear functional on~$C_b$.  By
  the Riesz-Markov theorem \cite[Theorem IV.18]{RS}, there is a finite measure~$m$ such that $$
  \psi(f) = \int_\Rl f(x)\, dm(x) $$
for any continuous function $f$ of compact support. Under such terms, the weak
  convergence claimed in the lemma, turns into the convergence $$
  \lim_{n \rightarrow \infty} \psi(h_n) = \psi(h) $$ or rather $$
  \lim_{n \rightarrow \infty} \int_\Rl h_n(x)\, dm(x) = \int_\Rl
  h(x)\, dm(x). $$ The latter can be seen via the classical dominated
  convergence theorem for Lebesgue integration.  The lemma is proved.
\end{proof}

\subsection*{Divided differences}
\label{sec:divdif}

Let~$x_0, x_1, \ldots \in \Rl$ and let~$f$ be a tempered distribution such that~ $f^{(k)} \in L^\infty$ for every $1\leq k\le m$. The divided difference
$f^{[ k]}$ is defined recursively as follows.

The divided difference of the zeroth order~$f^{[0]}$ is the
function~$f$ itself.  The divided difference of order~$k = 1, \ldots,
m$ is defined by
\begin{align*}
  f^{[k]} \left( x_0, x_1, \tilde x \right) :=
  \begin{cases}\frac
    { f^{[k - 1]} (x_0, \tilde x) - f^{[k - 1]}(x_1,
      \tilde x)}{x_0 - x_1}, & \text{if~$x_0
      \neq x_1$}, \\ \frac {d}{dx_1} f^{[k - 1]} (x_1,
    \tilde x), & \text{if~$x_0=x_1$},
  \end{cases}
\end{align*}
where~$\tilde x = \left(x_2, \ldots, x_k \right) \in \Rl^{k - 1}$. Note that
$f^{[k+1]}=(f^{[k]})^{[1]}$. We claim that the
function~$f^{[ k]}$ admits the following integral representation
\begin{equation}
  \label{DDintRep}
  f^{[ k]} (x_0,
  \ldots, x_{ k}) = \int_{S_{ k}} f^{({ k})} \left( s_0 x_0 + \ldots + s_{ k}
x_{ k}
  \right)\, d\sigma_{ k} {, \,\,\text{ for every }\,  k\le m}.
\end{equation}
In other words, the function~$f^{[k]}$ is an $k$-th order
polynomial integral momentum associated with the
polynomial~$Q \equiv 1$ and the $k$-th derivative~$h = f^{(k)}$.

If~$k = 1$, then the claim~(\ref{DDintRep}) is a simple
restatement of the fundamental theorem of calculus
with the substitution  $t=x_0-s_1x_0+s_1x_1$ as follows
$$\int_{S_1} f' \left( s_0 x_0 + s_1 x_1
  \right)\, d\sigma_1\stackrel{(\ref{minv})}{=}\int_0^1 f' \left( (1-s_1) x_0 + s_1
x_1
  \right)\,
  ds_1=$$
  $$=\left\{\begin{array}{ll}\frac{1}{x_1-x_0}\int_{x_0}^{x_1} f'
\left(t\right)\,
  dt,&x_0\neq x_1\\\int_0^1 f' \left( x_0 \right)\,
  ds_1, & x_0=x_1\end{array}\right.=
  \left\{\begin{array}{ll}\frac{f(x_1)-f(x_0)}{x_1-x_0},&x_0\neq x_1\\f'
\left( x_0 \right), &
  x_0=x_1\end{array}\right.=f^{[1]}(x_0,x_1).
$$
\\

For~$1<k\leq m$, we prove \eqref{DDintRep} via the method of mathematical induction.
Suppose that we have already established that
$$
f^{[k]} (x_0,  \ldots, x_k) = \int_{S_k} f^{(k)} \left( s_0 x_0 + \ldots + s_k 
x_k  \right)\, d\sigma_k, \text{ for all } k\le n<m.
$$
Let us prove the statement for $n+1.$ For $\tilde
x=(x_2,\ldots ,x_{n+1})$ denote $$f_{\tilde x}(x):=f^{[n]} (x,
  x_2,\ldots, x_{n+1})=\int_{S_n} f^{(n)} \left( s_0 x+s_1 x_2 + \ldots +
  s_{n}  x_{n+1}
  \right)\, d\sigma_n,$$ which is an $n$-th order integral momentum with the
function $h:=f^{(n)}$ and $Q\equiv 1.$
Now, it follows from Lemma~\ref{PIMandDD} below that
$$\psi(x_0,\ldots ,x_{n+1}):=f_{\tilde
x}^{[1]}(x_0,x_1)=f^{[n+1]}(x_0,\ldots ,x_{n+1})$$
is an $(n+1)$-th order integral momentum associated with the function
$h'=f^{(n+1)}$ and $Q\equiv 1$, that is
  $$f^{[n+1]}(x_0,\ldots ,x_{n+1})=\int_{S_{n+1}} f^{(n+1)} \left( s_0 x_0+ \ldots +
  s_{n+1}
  x_{n+1}
  \right)\, d\sigma_{n+1}.$$
In  other words, the claim  \eqref{DDintRep} also holds for $k=n+1$.

Immediate implications of Theorems~\ref{PIMpolyEst} and~\ref{PSStheorem}
for divided differences are as follows.

\begin{theorem}
  \label{DDImpl}
  \begin{enumerate}
  \item If~$f \in \tilde B^m_{\infty 1}$, then the
    operator~(\ref{MOIdef}) is bounded and $$ \left\| T_{f^{[m]}}
    \right\|_{\tilde p\to p} \leq \const \left\| f\right\|_{\tilde B^m_{\infty
1}}. $$

  \item If\/~$f^{(m)} \in C_b$ and if\/~$1 < p < \infty$, then the
    (modified) operator~$\hat T_{f^{[m]}}$ is bounded and $$ \left\|
      \hat T_{f^{[m]}} \right\|_{\tilde p\to p} \leq \const \left\| f^{(m)}
    \right\|_\infty. $$

  \end{enumerate}
\end{theorem}
When ~$f^{(m)} \in C_b$, we shall also consider the function~$\tilde f^{[m]}$ defined by setting
\begin{equation}\label{DDtilde}
\tilde f^{[m]}(x_0, \ldots, x_{m-1}) := g^{[m-1]}(x_0, \ldots, x_{m-1}),\ {\rm where}\ g :=
f'.
\end{equation}
It follows from \eqref{DDintRep} and definition \eqref{2stars} that the
function~$\tilde f^{[m]}$ is an $(m-1)$-th order
polynomial integral momentum associated with the function~$h =
f^{(m)}$ and the polynomial~$Q = 1$.

\subsection*{Perturbation of multiple operator integrals}

Let~$\phi = \phi_{m, h, Q}$ be a polynomial integral momentum
associated with a function~$h$ such that~$h' \in L^\infty$ and~$Q \in \Pl_m$.
For~$\tilde x = (x_2, \ldots, x_{m + 1})$ we set~$f_{\tilde x} (x) =
\phi(x, \tilde x)$.  We now consider the divided difference $$
\psi(x_0, \ldots, x_{m + 1}) := f_{\tilde x}^{[1]}(x_0, x_1). $$

\begin{lemma}
  \label{PIMandDD}
  The function~$\psi$ is the ~$(m + 1)$-th order integral momentum ~$ \phi_{m+1,
h', Q_1}$
  associated with the function~$h'$ and the polynomial~$Q_1 \in \Pl_{m+1}$
  given by $$ Q_1 (s_1, \ldots, s_{m+1}) = Q(s_2, \ldots, s_{m+1}). $$
\end{lemma}

\begin{proof}[Proof of Lemma~\ref{PIMandDD}]
  By definition \eqref{2stars} and taking the derivative,
  \begin{multline*}
    f'_{\tilde x} (x) = \int_{\scriptstyle s_2 + \ldots +
      s_{m+1} \leq 1 \atop \scriptstyle s_2, \ldots, s_{m+1} \geq 0}
    Q(\tilde s)\, s_0\, h'(s_0 x + s_2 x_2 + \ldots + s_{m + 1} x_{m +
      1})\, ds_2\, \ldots \, ds_{m + 1}, \\ \text{where}\ \ s_0 = 1 -
    s_2 - \ldots - s_{m + 1} \ \ \text{and}\ \ \tilde s = (s_2,
    \ldots, s_{m + 1}).
  \end{multline*}
  On the other hand, via the representation~(\ref{DDintRep}) for~$m =
  1$, $$ f^{[1]}_{\tilde x} (x_0, x_1) = \int_0^1 f'_{\tilde x} \left(
    l_0 x_0 + l_1 x_1 \right)\, dl_1,\ \ \text{where}\ \ l_0 = 1-
  l_1. $$ Combining the two,
  \begin{multline*}
    \psi(x_0, \ldots, x_{m + 1}) =
    \int_0^1 dl_1 \int_{\scriptstyle s_2 + \ldots + s_{m+1} \leq
      1 \atop \scriptstyle s_2, \ldots, s_{m+1} \geq 0} Q(\tilde s)\,
    s_0 h' \bigl( s_0 (l_0 x_0 + l_1 x_1) + \\ s_2 x_2 + \ldots + s_{m + 1}
    x_{m + 1} \bigr)\, ds_2\, \ldots \, ds_{m + 1}.
  \end{multline*}
  We next substitute the integration $(m + 1)$-tuple~$(l_1, s_2,
  \ldots, s_{m + 1})$ with $(m+1)$-tuple $(s'_1, \ldots, s'_{m + 1})$
  as follows $$ s'_1 = s_0 l_1 \ \ \text{and}\ \ s'_j = s_j,\ \ j = 2,
  \ldots, m+ 1. $$ Under such substitution, the integration domain $$
  0 \leq l_1 \leq 1 \ \ \text{and}\ \ s_2 + \ldots + s_{m + 1} \leq
  1,\ \ s_2, \ldots, s_{m + 1} \geq 0 $$ becomes the domain $$ s'_1 +
  \ldots + s'_{m + 1} \leq 1 \ \ \text{and}\ \ s'_1, \ldots, s'_{m +
    1} \geq 0; $$ $$s'_1 +
  \ldots + s'_{m + 1}
=s_0l_1+s_2+\ldots +s_{m+1}=(1-s_2-\ldots -s_{m+1})l_1+s_2+\ldots +s_{m+1}$$
  $$=l_1+(1-l_1)(s_2+\ldots +s_{m+1})\le
  l_1+1-l_1=1.$$ Computing the Jacobian $J$ of the substitution, we have

  \begin{equation*}
    J :=
    \begin{bmatrix}
      s_0 & 0 & 0 & \cdots  & 0 \\
      0 & 1 & 0 & \cdots  & 0 \\
      0 & 0 & 1 & \cdots  & 0 \\
      \vdots & \vdots & \vdots & \ddots  & \vdots \\
      0 & 0 & 0 & \cdots  & 1
    \end{bmatrix}=s_0 .
  \end{equation*}

  Observe also that if $$ s'_0 := 1 - s'_1 - \ldots - s'_{m + 1}, $$
  then~$s'_0 = 1-s_0l_1-s_2-\ldots -s_{m+1}=s_0-s_0l_1= s_0 l_0$.  Thus, we obtain
that
  \begin{multline*}
    \psi(x_0, \ldots, x_{m + 1}) =
    \int_{\scriptstyle s'_1 + \ldots + s'_{m + 1}\le 1 \atop \scriptstyle
      s'_1, \ldots, s'_{m + 1} \geq 0} Q(\tilde s)\, h'(s'_0 x_0 +
    \ldots + s'_{m + 1} x_{m + 1})\, ds'_1\, \ldots \, ds'_{m + 1}
    \\
    \quad=\int_{S_{m + 1}} Q_1 (\tilde s')\, h'(s'_0 x_0 + \ldots + s'_{m +
      1} x_{m + 1})\, d\sigma_{m +1 }, \\ \text{where}\ \ \tilde s' =
    (s'_1, \ldots, s'_{m + 1}).
  \end{multline*}
  That is, function~$\psi$ is a polynomial integral momentum.
\end{proof}

\begin{theorem}
  \label{PertubationFormula}
  Let~$\phi = \phi_{m, h, Q}$ and~$\psi = \phi_{m+1, h', Q_1}$ be from Lemma~\ref{PIMandDD}.  Let~$A, B$ are
  bounded self-adjoint operators.  If~$h \in \tilde B^1_{\infty 1}$ and
  if~$\tilde H = (H_1, \ldots, H_m)$, then $$ T_{\phi}^{A, \tilde H}
  (V_1, \ldots, V_m) - T_\phi^{B, \tilde H} (V_1, \ldots, V_m) =
  T_\psi^{A, B, \tilde H} (A - B, V_1, \ldots, V_m). $$
\end{theorem}

\begin{proof}[Proof of Theorem~\ref{PertubationFormula}]
Let us denote
  \begin{multline*}
\psi_1 (x_0, \ldots,
    x_{m + 1}) := x_0\, \psi(x_0, \ldots, x_{m + 1});\
    \psi_2 (x_0, \ldots, x_{m + 1}) := x_1\, \psi(x_0, \ldots, x_{m +
      1}); \\ \phi_1 (x_0, \ldots, x_{m + 1}) := \phi(x_0, x_2, \ldots,
    x_{m + 1});\ \phi_2(x_0, \ldots, x_{m+1}) :=
    \phi(x_1, \ldots, x_{m + 1}).
  \end{multline*}
We claim that $$ \psi_1 - \psi_2 = \phi_1 - \phi_2. $$
% To see the claim, recall from the proof of Lemma
% \ref{PIMandDD} that if $f_{\tilde x}(x)=\phi(x,\tilde x)$, then
% $$\psi(x_0,\ldots ,x_{m+1})=f^{[1]}_{\tilde x}(x_0,x_1)= \int_0^1
% f'_{\tilde x} \left(
%     l_0 x_0 + l_1 x_1 \right)\, dl_1,\ \ \text{where}\ \ l_0 = 1-
%   l_1.$$
% Substituting  $t=x_0-l_1x_0+l_1x_1$, we obtain the claim
% \begin{align*}(\psi_1&-\psi_2)(x_0,\ldots ,x_{m+1})
%=(x_0-x_1)\psi(x_0,\ldots ,x_{m+1})\\&=(x_0-x_1)\int_0^1
% f'_{\tilde x} \left( x_0-l_1 x_0 + l_1 x_1 \right)\, dl_1\\&=-\int_{x_0}^{x_1}
% f'_{\tilde x} \left(t \right)\, dt=f_{\tilde x}(x_0)-f_{\tilde
% x}(x_1)\\&=\phi(x_0,\tilde x)-\phi(x_1,\tilde
% x)=\phi_1(x_0,\ldots ,x_{m+1})-\phi_2(x_0,\ldots ,x_{m+1}).
% \end{align*}
To see the claim simply set $\tilde x:=(x_2,\dots, x_{m+1})$ and recall from the
definitions that 
$$
\psi(x_0,\ldots ,x_{m+1})(x_0-x_1)=f(x_0,\tilde x)-f(x_1,\tilde x)=\phi(x_0,\tilde
x)-\phi(x_1,\tilde x). 
$$

Note also that, since $h \in \tilde B^1_{\infty 1}$,
by Theorem \ref{PIMpolyEst} and Lemma \ref{PIMandDD}, the operator $T_\psi$ is well-defined
as well as the operators $T_{\psi_1 - \psi_2}$ and $T_{\phi_1 -
\phi_2}$, which are well-defined and satisfy $T_{\psi_1 -
\psi_2}=T_{\psi_1}-T_{\psi_2}$ and $T_{\phi_1 -
\phi_2}=T_{\phi_1}-T_{\phi_2}$ due to \cite[Proposition
4.10]{ACDS}.

Letting~$\tilde V = (V_1, \ldots, V_m)$ and using~(\ref{TphiHomomorphismI}), we
then have
  $$\begin{array}{rl}
    T_{\psi}^{A, B, \tilde H} (A - B, \tilde V)&  = T_{\psi}^{A, B,
      \tilde H} (A, \tilde V) - T_{\psi}^{A, B, \tilde H} (B, \tilde
    V) \\ &= T_{\psi_1}^{A, B, \tilde H} (1, \tilde V) - T_{\psi_2}^{A,
      B, \tilde H} (1, \tilde V) = T_{\psi_1 - \psi_2}^{A, B, \tilde
      H} (1, \tilde V) \\ & = T_{\phi_1 - \phi_2}^{A, B, \tilde H} (1,
    \tilde V) = T_{\phi_1}^{A, B, \tilde H}(1, \tilde V) -
    T_{\phi_2}^{A, B, \tilde H} (1, \tilde V) \\&= T_\phi^{A, \tilde
      H} (\tilde V) - T_\phi^{B, \tilde H} (\tilde V).
  \end{array}$$
\end{proof}

\subsection*{H\"older type estimates for polynomial integral momenta}

In this section, we fix a polynomial integral momentum ~$\phi = \phi_{m, h, Q}$ 
associated with a polynomial~$Q \in
\Pl_m$ and a function~$h \in L^\infty$.  Let also~$\tilde H = (H_1,
\ldots, H_m)$ and let~$V_j \in \cS^{p_j}$, $1\leq p_j\leq \infty$, $j = 1, \ldots, m$ be
fixed.  For a self-adjoint bounded operator~$A$, we shall consider the
mapping $$ F: A \mapsto F(A) := T_\phi^{A, \tilde H} (V_1, \ldots,
V_m). $$

In this section we shall establish H\"older estimates for the
mapping~$F$.  In the special case~$h=f$, $m=0$, $Q\equiv 1$ the Holder
properties of the mapping~$F$ were studied in~\cite[\S5]{AlPe-Sp-2010}.  This
section extends the technique of~\cite[\S5]{AlPe-Sp-2010} to the general 
mappings~$F$.  It should be pointed out that a
vital ingredient in our extension (even when $m=1$) is supplied by
Theorem~\ref{PSStheorem} (see the estimate
concerning the element~$Q_N$ in the proof of Theorem~\ref{HolderInWeak} below).

Recall that~$s_k(U)$ stands for the $k$-th singular number
associated with a compact operator~$U$. The symbol $s(U)$ stands for the sequence $\{s_k(U)\}_{k\ge 1}$. For the purposes of this
section, we introduce the following truncated norm $$ \left\| U
\right\|_{p, \nu} = \left( \sum_{k = 1}^\nu \big( s_k(U) \big)^p
\right)^{\frac
  1p}. $$

The theorem below  estimates the singular values of the operator~$F(A) -
F(B)$.

\begin{theorem}
  \label{HolderInWeak}
 Assume that~$A- B \in \cS^{p_0}$, that\/~$\sum_{j = 0}^m \frac 1{p_j} \leq 1$
  and set $$ U := F(A)-F(B)=T_{\phi}^{A, \tilde H}(V_1, \ldots, V_m) -
  T_{\phi}^{B, \tilde H} (V_1, \ldots, V_m). $$ If\/~$h \in
  \Lambda_\alpha\cap \tilde B^0_{\infty1}$, for some~$0 < \alpha <
1$, then $$ s_k(U) \leq
  \const\, k^{- \frac {1} p}\, \left\| h
  \right\|_{\Lambda_\alpha}\, \left\| A - B\right\|^\alpha_{p_0,
    \nu}\, \left\| V_1 \right\|_{p_1} \cdot \ldots \cdot \left\| V_m
  \right\|_{p_m}, $$
where $\frac 1p = \frac
  \alpha{p_0} + \sum_{j = 1}^m \frac 1{p_j}$ and $\nu\ge \frac k2.$
\end{theorem}

\begin{proof}[Proof of Theorem~\ref{HolderInWeak}]
  Assume for simplicity that $$ \left\| V_1 \right\|_{p_1} = \ldots =
  \left\| V_m\right\|_{p_m} = 1. $$ Let~$W_n$ be the Schwartz function
  from the definition of the Besov spaces (see~(\ref{WnFuncs})).  For every
$n\in \Z$, we set
$$h_n := W_n * h,\  \phi_n := \phi_{m, h_n, Q},\ U_n :=
  T_{\phi_n}^{A, \tilde H}(V_1, \ldots, V_m) - T_{\phi_n}^{B, \tilde
    H} (V_1, \ldots, V_m). $$
Here, we justify the existence of the operator  $ T_{\phi_n}$ by
appealing to Theorem \ref{PIMpolyEst}(i). We fix~$N \in \Z$
 (the choice of~$N$ will be specified later) and set
$$ R_N := \sum_{n \leq N} U_n \ \  \text{and}\ \ Q_N: = \sum_{n > N} U_n. $$
We claim  that $$U=R_N+Q_N.$$
% To see the claim we firstly recall the
% footnote at page 5 and the fact that $\mathcal F(f*g)=\mathcal
% F(f) \mathcal F(g),$ where $\mathcal F$ here is Fourier transform.
% Combining these facts, we have
% $$\mathcal F\big(\sum_{n\in\mathbb Z}(h*W_n)\big)=\sum_{n\in\mathbb Z}\mathcalF(h)\mathcal F(W_n)=
%  \mathcal F(h)\sum_{n\in\mathbb Z}\hat W_n=\mathcal F(h).$$
%
% (To justify the equalities above, it is sufficient to observe that
% the assumption $h\in \Lambda_\alpha$ guarantees that
% $\sup_{n\in\mathbb Z}2^{n\alpha}\|h*W_n\|_\infty<\infty$  (
% {\color{red} nuzhno proverit' dostatochno etix ob'yasnenii i verny
% li oni}. {  I checked. Now I see that I was wrong. Give
% me more time.}

Since $h \in \tilde B^0_{\infty1}$, it follows from the definition
of the norm of the Besov space $\tilde B^0_{\infty1}$ that the
series $\sum_{n\in\mathbb Z} h_n$ converges uniformly. Noting that
the latter series consists of continuous (in fact smooth and
rapidly decreasing at $\infty$) functions, we conclude that it
also converges in the space of all continuous functions on
$\mathbb{R}$. It follows (see also \cite{AlPe-Sp-2010}) that $\sum_{n\in\mathbb
Z}(h*W_n)=h$ and
so for $\tilde s=(s_1,\ldots ,s_m)$ we have
\begin{align*}
 \phi(x_0,\ldots ,x_m)&=\int_{S_m} Q\left(
  \tilde s \right)\, h \left( \sum_{j = 0}^m s_j x_j \right)
\, d\sigma_m=\int_{S_m} Q\left(
  \tilde s \right)\, \sum_{n\in\mathbb Z}(h*W_n) \left( \sum_{j = 0}^m s_j x_j
\right)
\, d\sigma_m\\  &=
   \int_{S_m} Q\left(
  \tilde s \right)\, \sum_{n\in\mathbb Z}h_n \left( \sum_{j = 0}^m s_j x_j
\right)
\, d\sigma_m=\sum_{n\in\mathbb Z}\int_{S_m} Q\left(
  \tilde s \right)\, h_n \left( \sum_{j = 0}^m s_j x_j \right)
\, d\sigma_m\\ &=\sum_{n\in\mathbb Z}\phi_n(x_0,\ldots ,x_m).
\end{align*}
Now, we arrive at the claim as follows
\begin{align*}U &= T_{\phi}^{A, \tilde H}(V_1, \ldots, V_m) -
  T_{\phi}^{B, \tilde H} (V_1, \ldots, V_m)\\ &=T_{\sum_{n\in\mathbb
Z}\phi_n}^{A, \tilde H}(V_1, \ldots, V_m) - T_{\sum_{n\in\mathbb Z}\phi_n}^{B,
\tilde
    H} (V_1, \ldots, V_m)\\&=
\sum_{n\in\mathbb Z}(T_{\phi_n}^{A, \tilde H}(V_1, \ldots, V_m)
-T_{\phi_{n}}^{B, \tilde H} (V_1, \ldots, V_m))\\ &=\sum_{n\in\mathbb
Z}U_n=R_N+Q_N.
\end{align*}
Here, the step from the second to the third line above is justified as follows.
Firstly, we note that the series $\phi=\sum_{n\in \Z} \phi_n$
converges also in the norm  $\|\cdot \|_{\cA_m}$. This convergence follows from
the already used above fact that $\sum _{n\in \mathbb Z}\|h_n\|_\infty<\infty$ combined with
Theorem \ref{PIMpolyEst}(i). Hence, appealing to \eqref{TphiBasicEst}, we infer
that $T_\phi=\sum_{n\in \Z} T_{\phi_n}$
(in the sense of the strong operator topology).

Observing the following  elementary properties of singular values $$ s_k(U + V)
\leq s_{\frac
    k2} (U) + s_{\frac k2}(V) \ \ \text{and}\ \ s_k(U) \leq k^{- \frac
    1p} \left\| U \right\|_{p, \nu},\ \ k \leq \nu, $$ we see
  \begin{equation}
    \label{HolderInWeakTmp}
    s_k (U) \leq s_{\frac k2} (R_N) + s_{\frac k2} (Q_N)
    \leq {\Big(\frac k2\Big)}^{-\frac 1r} \left\| R_N\right\|_{r, \nu} +
{\Big(\frac k2\Big)}^{-\frac 1{r_0}}
    \left\| Q_{N} \right\|_{r_0}, \quad {\frac k2\le
    \nu,}
  \end{equation}
where~$r_0^{-1} = \sum_{j = 1}^m p_j^{-1}$ and~$r^{-1} = p_0^{-1} +
  r_0^{-1}$.  We now estimate~$R_N$ and~$Q_N$ separately.

  We estimate~$R_N$ as follows.  Observe that by
  Theorem~\ref{PertubationFormula}, $$ U_n = T_{\psi_n}^{A, B, \tilde H} (A-B,
  V_1, \ldots, V_m), $$ where~$\psi_n = \phi_{m+1, h'_n, Q}$ is the
  polynomial integral momentum of order~$m + 1$ associated with the
  function~$h'_n = h' * W_n$, where we view $h'$ as a generalized function
(the preceding equality follows immediately from the definition $h_n := W_n *
h$).
Since $\widehat{h'_n}(\xi)=2\pi i\xi \widehat{h}_n(\xi)$, we readily infer from 
Lemma~\ref{PIMpolyLemma}
  part~(\ref{PIMpolyLemma:hm}) with $N=n$ and $m=1$ that
\begin{equation}\label{temporary3} \left\| h'_n
  \right\|_\infty \leq \const\, 2^n \left\| h_{ n} \right\|_\infty.
  \end{equation}
It is also known as a combination of
\cite[Proposition 7]{Stein} and \cite[Corollary 2]{Triebel-1982}
that
  \begin{equation}
    \label{HolderHarmonicDef}
    2^{\alpha n} \left\| h_n \right\|_\infty \leq
    \const\,\left\| h \right\|_{\Lambda_\alpha}.
  \end{equation}
Combining \eqref{HolderHarmonicDef} with \eqref{TphiBasicEst} and
Theorem~\ref{PIMpolyEst} part~(i) we see that
\begin{align*}
\| U_n\|_{r,\nu}&=\|T_{\psi_n}^{A, B,\tilde H} (A-B,
  V_1, \ldots, V_m)
\|_{r,\nu}\\&\le\|T_{\psi_n}^{A, B,\tilde H}\|_{\tilde p\to
r}\|A-B\|_{p_0,\nu}\|V_1\|_{p_1}\ldots \|V_m\|_{p_m}\\&\le
\const\, 2^n\, \left\| h_n
  \right\|_\infty\|A-B\|_{p_0,\nu}\\ &\leq \const\,
  2^{(1-\alpha)n} \, \left\| h \right\|_{\Lambda_\alpha}\, \left\| A -
    B \right\|_{p_0, \nu}.
 \end{align*}
Noting that $\sum_{n\le
N}2^{(1-\alpha)n}=
%2^{(1-\alpha)N}+2^{(1-\alpha)(N-1)}+\ldots =\frac{2^{(1-\alpha)N}}
%{\color{red} {1-\frac{1}{2^{1-\alpha}}} ???}=
\const\,2^{(1-\alpha)N},$ we obtain $$ \left\| R_N \right\|_{r, \nu}
  \leq \sum_{n \leq N} \left\| U_n \right\|_{r, \nu} \leq \const\,
  2^{(1-\alpha)N}\, \left\| h \right\|_{\Lambda_\alpha}\, \left\| A -
    B\right\|_{p_0, \nu}. $$
In order to estimate~$Q_N$, we combine Theorem~\ref{PSStheorem} (see the
comments following the statement of Theorem ~\ref{PSStheorem},
which explain why we are in a position to identify operators $T_\phi$ and $\hat
T_\phi$) and~(\ref{HolderHarmonicDef}) as follows
\begin{align*}
    \left\| U_n \right\|_{r_0} &\leq \left\| T^{A, \tilde
        H}_{\phi_n} (V_1, \ldots, V_m) \right\|_{r_0} + \left\| T^{B,
        \tilde H}_{\phi_n} (V_1, \ldots, V_m) \right\|_{r_0} \\ &\leq \const
    \left\| h_n \right\|_\infty \leq \const 2^{-\alpha n}\, \left\| h
    \right\|_{\Lambda_\alpha}.
\end{align*}

Consequently, $$ \left\| Q_N
  \right\|_{r_0} \leq \sum_{n > N} \left\| U_n \right\|_{r_0} \leq
  \const\, 2^{-\alpha N}\, \left\| h \right\|_{\Lambda_\alpha}. $$

  Returning back to~(\ref{HolderInWeakTmp}), we arrive at $$ s_k(U)
  \leq \const\, 2^{-\alpha N}\, \left\| h \right\|_{\Lambda_\alpha}\,
  \Big(\frac k2\Big)^{-\frac 1r_0}\, \left[ \Big(\frac k2\Big)^{-\frac 1{p_0}}
2^{N}\, \left\| A - B
    \right\|_{p_0, \nu} + 1 \right]. $$ The proof can now be finished by
  choosing~$N \in \Z$ such that

\begin{equation}\label{twosided} 2^{-N - 1} \leq \Big(\frac k2\Big)^{- \frac
    1{p_0}}\, \left\| A - B \right\|_{p_0, \nu} < 2^{-N}.
\end{equation}

% {  (*) Look at the display above. If $N\in \mathbb N$,
% then you can choose such $N$ as above only in case $0\le
% \Big(\frac k2\Big)^{- \frac
%     1{p_0}}\, \left\| A - B \right\|_{p_0, \nu}< 1.$ Moreover, we used around here the fact that $\sum_{n\in\mathbb Z}\hat W_n=1.$   }

Indeed, suppose $N$ is such as above. Then rewriting the preceding estimate, we
have
\begin{align*} s_k(U)
  &\leq \const\, 2^{-\alpha N}\, \left\| h \right\|_{\Lambda_\alpha}\,
  \Big(\frac k2\Big)^{-\frac 1r_0}\, \left( \Big(\frac k2\Big)^{-\frac 1{p_0}}
\, \left\| A - B
    \right\|_{p_0, \nu} 2^{N} + 1 \right)\\&\le \const\, 2^{-\alpha N}\, \left\|
h \right\|_{\Lambda_\alpha}\,
  \Big(\frac k2\Big)^{-\frac 1r_0}\, \left( 2^{-N} 2^{N} + 1 \right)
  \\&=\const\, 2^{1+\alpha}\,2^{\alpha (-N-1)}\, \left\| h
\right\|_{\Lambda_\alpha}\,
  \Big(\frac k2\Big)^{-\frac 1r_0}\\& \le \const\, 2^{1+\alpha}\Big(\frac
k2\Big)^{- \frac
    \alpha{p_0}}\, \left\| A - B \right\|_{p_0, \nu}^\alpha\left\| h
\right\|_{\Lambda_\alpha}\,
  \Big(\frac k2\Big)^{-\frac 1r_0}\\&=\const\, k^{- \frac
    1{p}}\, \left\| A - B \right\|_{p_0, \nu}^\alpha\left\| h
\right\|_{\Lambda_\alpha},
\end{align*}
where we used firstly the right hand side from \eqref{twosided} and then its
left hand side, and, in the last step, the equalities $\frac 1p = \frac
  \alpha{p_0} + \sum_{j = 1}^m \frac 1{p_j}$ and $\frac 1{r_0} = \sum_{j = 1}^m
\frac
  1{p_j}$.
\end{proof}

Recall that~$S^{p, \infty}$, $1 \leq p < \infty$ stands for the weak
Schatten-von Neumann quasi-normed ideal defined by the relation $$
\left\| U \right\|_{p, \infty} := \sup_{k \geq 1} k^{\frac 1p} s_k(U)
< + \infty. $$ Letting~$\nu \rightarrow \infty$, we also have the
corollary.

\begin{corollary}
  \label{HolderInWeakI}
  In the setting of Theorem~\ref{HolderInWeak}, we have $U \in
  S^{p, \infty}$ and $$ \left\| U \right\|_{p, \infty} \leq \const
  \left\| h \right\|_{\Lambda_\alpha}\, \left\| A -
    B\right\|^\alpha_{p_0}\, \left\| V_1 \right\|_{p_1} \cdot \ldots
  \cdot \left\| V_m \right\|_{p_m}, $$ where~$\frac 1p = \frac
  \alpha{p_0} + \sum_{j = 1}^m \frac 1{p_j}$.
\end{corollary}

Finally, the H\"older estimate for the mapping~$F$ is given below.

\begin{theorem}
  \label{HolderInFull}
  In the setting of Theorem~\ref{HolderInWeak}, if\/~$\sum_{j = 0}^m
  \frac 1 {p_j} < 1$, then~$U \in \cS^p$ and $$ \left\|
    U \right\|_{p} \leq \const\, \left\| h \right\|_{\Lambda_\alpha}\,
  \left\| A - B \right\|_{p_0}^\alpha\, \left\| V_1 \right\|_{p_1}
  \cdot \ldots \cdot \left\| V_m \right\|_{p_m}, $$ where~$\frac 1p =
  \frac \alpha{p_0} + \sum_{j = 1}^m \frac 1{p_j}$.
\end{theorem}

\begin{proof}[Proof of Theorem~\ref{HolderInFull}]
We shall consider two mutually exclusive situations. Firstly, we assume that there is~$p_j$ ($1 \leq j \leq m$) such that~$p_j <
  \infty$.  In this case the claim follows from the real interpolation
  method directly.  Indeed, assume for simplicity that~$p_1 < \infty$.
  Let~$A-B \in \cS^{p_0}$ and~$V_j \in \cS^{p_j}$, $j = 2, \ldots m$
  be fixed such that $$ \left\| V_j \right\|_{p_j} = 1,\ \ j = 2,
  \ldots, m. $$ Let $$ T(V) = T_{\phi}^{A, \tilde H}(V, V_2, \ldots,
  V_m) - T_{\phi}^{B, \tilde H} (V, V_2, \ldots, V_m). $$
Applying Corollary~\ref{HolderInWeakI} with ~$\frac 1 {r_0} = \frac \alpha{p_0}
+ \sum_{j = 2}^m \frac
  1{p_j}$ and with~$\frac {1}{r_1} = \frac {1}{r_0} + \frac {1}{\tilde p_1}$,
~$\frac {1}{\tilde p_1}: = 1 - \sum_{j = 2}^m \frac {1}{p_j}$, we have
respectively
  \begin{multline*}
    \left\| T(V) \right\|_{r_0, \infty} \leq \const\,{\left\| h
\right\|_{\Lambda_\alpha}} \left\| A - B
    \right\|_{p_0}^\alpha\, \left\| V \right\|_{\infty} \ \
    \text{and}\\ \left\| T(V) \right\|_{r_1, \infty} \leq \const\,
    {\left\| h \right\|_{\Lambda_\alpha}}\left\| A - B \right\|_{p_0}^\alpha\,
\left\| V \right\|_{\tilde
      p_1}.
\end{multline*}
 Observe that $\frac 1{p_1}<\frac1{\tilde p_1}$ and hence, $0<\theta:=\frac
{\tilde p_1}{p_1}<1$.
  Applying the real interpolation method~$\left[\cdot,\cdot \right]_{\theta, p}$
to the quasi-Banach pair $(\cS^{r_0, \infty}, \cS^{r_1, \infty})$, we conclude
the proof.

Now, we assume now that~$p_j = \infty$ for every~$j = 1, \ldots, m$.  In
  this case, the proof is similar to the argument
  used in~\cite[Theorem~5.8]{AlPe-Sp-2010}.  Assume for simplicity
  that $$ \left\| V_j \right\|_{\infty} = 1,\ \ j = 1, \ldots, m. $$
Applying Theorem ~\ref{HolderInWeak} with $p_0=1$ (so $\frac 1p=\alpha$) and $p_j=\infty$ for all $1\le j\le m,$ we have
$$
s_k(U)\leq \const \left\| h \right\|_{\Lambda_\alpha}\left(\frac {1}{k}\right)^\alpha \|A-B\|_{1,\nu}^\alpha,\ \forall \nu\ge \frac k2,
$$
or equivalently, 
$$
s_k^{\frac{1}{\alpha}}(U)\leq \const\, \left\| h \right\|_{\Lambda_\alpha}^{\frac{1}{\alpha}}\,\frac {1}{k} \|A-B\|_{1,\nu},\ \forall \nu\ge \frac k2.
$$
In particular, setting $\nu=k$, we obtain
$$ s_k(\left| U \right|^{\frac{1}{\alpha}}) \leq \const\,\left\| h \right\|_{\Lambda_\alpha}^{\frac{1}{\alpha}}\,\frac
1k\sum _{n=1}^k s_n(A-B).$$
Considering C\' esaro operator $C$ on the space 
$l_\infty$ of all bounded sequences $x=\{x_n\}_{n\ge 1}$ given by the formula
$$(Cx)_k:=\frac 1k\sum _{n=1}^k x_n, \ k\ge 1,
$$
we may interpret  the preceding estimate as 
$$
s(|U|^{\frac{1}{\alpha}})\leq \const\,\left\| h \right\|_{\Lambda_\alpha}^{\frac{1}{\alpha}} Cs(A-B).
$$
Recalling that the operator $C$ maps the space $l_{p_0}$ into itself (for every $1<p_0\leq \infty$) and that, by the assumption, $A-B\in \cS^{p_0}$, we obtain that
$  Cs(A-B)\in l_{p_0}$ and therefore $|U|^{\frac{1}{\alpha}}\in \cS^{p_0}$, or equivalently $|U|^{\frac{p_0}{\alpha}}=|U|^p\in \cS^{1}$ 
(indeed, in our current setting we have $\alpha p=p_0$) and furthermore
$$
\||U|^{p}\|_1\leq \const\, \left\| h \right\|_{\Lambda_\alpha}^{\frac{p_0}{\alpha}}\,\|Cs(A-B)\|_{p_0}^{p_0}
\leq \const\, \left\| h \right\|_{\Lambda_\alpha}^{\frac{p_0}{\alpha}}\,\|A-B\|_{p_0}^{p_0},
$$
which is equivalent to the claim.
%Thus, obviously, $U\in \cS^p$.
\end{proof}

{ \begin{remark}
  \label{HolderInFullAlphaOne}
We observe that the assertion of Theorem~\ref{HolderInFull} also holds when~$\alpha = 1$ (in this case, we speak of Lipschitz functions rather than 
H\" older functions with exponent $\alpha$).  However, the proof of this case is based on
  totally different ideas.  In fact, this case is justified by
  Theorem~\ref{PSStheorem}.
\end{remark}
}
\section{Proof of the main result}
\label{sec:taylor-proof}

In this section, we consider $\cS^p$, $1 \leq p \leq \infty$ as a Banach space over the field $\Rl$ of real
numbers.

\begin{theorem}
  \label{TaylorExpansion}
  If $1 < p < \infty$ and if~$m \in \N$ is such
  that~$m <p \leq m + 1$, then for every ~$H \in \cS^p$, $\|H\|_p\leq 1$, there exist bounded
symmetric polylinear forms
  \begin{equation*}
    \dfun{1} : \cS^p \mapsto \Rl,\ \
    \dfun{2} : \cS^p \times \cS^p \mapsto \Rl,\ \
    \ldots,\ \
    \dfun{m}: \underbrace{\cS^p \times \ldots \times
      \cS^p}_{\text{$m$-times}} \mapsto \Rl
  \end{equation*}
  such that
  \begin{equation}
    \label{eq:TaylorExpansion1}
    \left\| H + V \right\|_p^p - \left\| H \right\|_p^p -
    \sum_{k = 1}^m \dfun{k} \Bigl(\underbrace{V, \ldots,
      V}_{\text{$k$-times}}\Bigr) = O(\left\| V \right\|_p^p),
  \end{equation}
  where~$V \in \cS^p$ and~$\left\| V \right\|_p \rightarrow 0$.
\end{theorem}

Observe that, without loss of generality, the above theorem needs only
a proof for the special case when~$H$ and~$V$ are self-adjoint
operators.  Indeed, let us assume that the theorem is proved in the
self-adjoint case, that is for every self-adjoint operators ~$H$ and~$V$ from
$\cS^p$ the existence of $\delta^{(k)}_H$'s satisfying
\eqref{eq:TaylorExpansion1} is established.
Fixing an infinite projection on ~${\mathcal H}$ with the infinite
orthocomplement, we may represent an arbitrary element
$X\in S^p$ as
\begin{equation*}
X =
\begin{pmatrix}
  X_{11} & X_{12} \\
  X_{21} & X_{22}
\end{pmatrix}
\end{equation*}
with $X_{ij}\in \cS^p$, $1\leq i,j\leq 2$. Furthermore, setting for an arbitrary
$X\in \cS^p$
\begin{equation*}
\alpha(X) =\frac {1}{2^{1/p}}
\begin{pmatrix}
  0 & X \\
  X^* & 0
\end{pmatrix}
\end{equation*}
we see that $\alpha$ is an isometrical embedding of $\cS^p$ into
itself (in fact, into a (real) Banach subspace of $\cS^p$
consisting of self-adjoint operators). Finally, for arbitrary operators
$H,V\in \cS^p$, we set  
$$\delta^{(k)}_{H}(V,\ldots ,V):=\frac
12\delta^{(k)}_{\alpha(H)}(\alpha(V),\ldots ,\alpha(V)),\ 1\leq k\leq
m.$$ It is trivial that  $\delta^{(k)}_H$'s are bounded symmetric
polylinear forms satisfying \eqref{eq:TaylorExpansion1}.

So, from now and until the end of the proof, we assume that~$H$ and~$V$ are self-adjoint operators such that 
$$ \left\| H \right\|_\infty \leq 1 \ \ \text{and}\ \ \left\| V
\right\|_\infty \leq 1. $$
Let~$f_p$ be the ~\lq\lq smoothed\rq\rq~ function~$\left| \cdot \right|^p$,
that is $f_p$ is a $C^\infty$ compactly supported function on $\Rl\backslash\{0\}$ such that $f_p(x) =
\left| x \right|^p$ for all~$\left| x
\right| \leq 2$.
Clearly, $$ \left\| H \right\|_p^p = \tr \left( f_p( H ) \right) \
\ \text{and}\ \ \left\| H + V \right\|_p^p = \tr \left( f_p(H + V)
\right). $$

\subsection*{The definition of functionals~$\delta_H^{(k)}$, $1\leq k\leq m$.}

We shall explicitly define the functionals~$\delta_H^{(k)}$'s
from~(\ref{eq:TaylorExpansion1}) in~(\ref{deltaFnDef}) below.  However,
given that the definition in~(\ref{deltaFnDef}) is rather complex, we
shall first give some guiding explanations.

We observe first that if~$\delta^{(k)}_H$ is a set of functionals from
the expansion~(\ref{eq:TaylorExpansion1}), then it is readily seen
that $$ \frac {d^k}{dt^k} \left[ \tr \left( f_p(H_t) \right) \right]
\Bigr|_{t = 0} = k!\, \delta^{(k)}_{H} \Bigl( \underbrace{V, \ldots,
  V}_{\text{$k$-times}} \Bigr), $$
where~$H_t = H + tV$.  On the other
hand, it is known from~\cite[Theorem
5.6]{Peller2006} and ~\cite[Theorem 5.7]{ACDS} that
$$ \frac {d^{k}}{dt^k} \left[ f_p(H_t) \right] = k!\,
T_{f_p^{[k]}}^{H_t} \Bigl( \underbrace{V, \ldots,
V}_{\text{$k$-times}} \Bigr), $$ where, for a function $\phi\in \cC_m$, we used the
abbreviation
\begin{equation}\label{abbreviation}
T^H_{\phi} = T^{\tilde H}_{\phi},\ \tilde
H = \Bigl( \underbrace{H, \ldots, H}_{\text{$m+1$-times}} \Bigr).
\end{equation}
Comparing the two identities above, it seems natural to suggest
the following definition for the functionals~$\delta^{(k)}_H$:
$$\delta_H^{(k)} (V_1, \ldots, V_k) = \tr \left( T_{f^{[k]}_p}^H
\left( V_1, \ldots, V_k \right) \right),\quad
V_1,\ldots ,V_k\in\mathcal S_p.$$ However, this suggestion is flawed
since a combination of Lemma~\ref{HolderToBesov} and
Theorem~\ref{DDImpl} (i) yields only that
$\|T_{f^{[k]}_p}^H\|_{\tilde p\to \frac
pk}\le\const\,\|f_p\|_{\tilde B_{\infty 1}^k},$ where $\tilde
p=(\underbrace {p, p,\dots, p}_{\text{$k$-times}})$, in particular
$$ U_k := T_{f^{[k]}_p}^H \left( V_1, \ldots, V_k
\right) \in \cS^{\frac pk},\ \ k = 1, \ldots, m. $$
In other words, it is not known (and not clear) whether~$U_k \in \cS^1$.

To circumvent this difficultly, we use the approach implicitly suggested
in~\cite[Lemma 2.2]{PSS-SSF2}. This approach is based on the identity.
$$ \tr \Big(
  T_{f^{[k]}}^{ H} (\underbrace{V, \ldots, V}_{\text{$k$-times}}) \Big) = \frac 1k\, \tr \Big( V
  \cdot T_{\tilde f^{[k]}}^{ H}(\underbrace{V, \ldots, V}_{\text{$k-1$-times}}) \Big), $$
   where~$\tilde f^{[k]}$ is
the polylinear integral momentum defined in~(\ref{DDtilde}), that
is for ~$f^{(m)} \in C_b$,
$$
\tilde f^{[m]}(x_0, \ldots, x_{m-1}) = g^{[m-1]}(x_0, \ldots,
x_{m-1}),\ \ g = f'.
$$
%and on the left hand side $\tilde H=(\underbrace{H,\ldots, H}_{k+1-\text{times}}$), on the right hand side $\tilde H=(\underbrace{H, \ldots, H}_{k-\text{times}}$). 
%We shall not further highlight this difference. 
Using the identity above as a guidance and setting aside for a
moment the question why the operator  $V_1
  \cdot T_{\tilde f^{[k]}}^H(V_2, \ldots, V_k)$ (see below) belongs to the trace class
$\cS^1$ for every $k=2,\dots, m$,  we now
explicitly define the functional~$\delta^{[k]}_H$ as follows
\begin{equation}
  \label{deltaFnDef_square}
  \delta_H^{[k]}\left( V_1, \ldots,
    V_k \right) = \left\{\begin{array}{cl}{\tr \left( V_1 \,f_p'(H) \right),}
    &{ k=1}\\
    \frac 1k \, \tr \left( V_1 \cdot T_{\tilde f^{[k]}_p}^H (V_2,
    \ldots, V_k) \right), & 1<k\le m  \end{array}\right.
\end{equation}
The definition above is crucially important for the proof. In the next two
subsections we shall confirm that for every $1\leq k\leq m$ the
functional ~$\delta^{[k]}_H$ is well defined and
satisfies all the properties required in Theorem
\ref{TaylorExpansion} (excepting the symmetricity). However, the
functionals ~$\delta^{[k]}_H$ are not symmetric. To obtain
symmetric functionals satisfying all the requirements of Theorem
\ref{TaylorExpansion}, we resort to the standard symmetrisation
trick (see e.g. \cite[Section 40]{LS-1961}) by setting
\begin{equation}
  \label{deltaFnDef}
  \delta_H^{(k)}\left( V_1, \ldots,
    V_k \right) := \frac{1}{k!} \sum_\sigma \delta_H^{[k]} ( V_{\sigma(1)},
V_{\sigma(2)},\dots,V_{\sigma(k)} )
\end{equation}
where the sum is taken over all permutations
$\sigma(1), \sigma(2),\dots, \sigma(k)$ of the indices
$1,2,\dots,k$. It is trivial to verify that the functionals
~$\delta^{(k)}_H$'s satisfy already all the requirements of
Theorem \ref{TaylorExpansion} as soon as such a verification is
firstly performed for the functionals ~$\delta^{[k]}_H$'s.

Such a verification for functionals~$\delta^{[k]}_H$'s (including their
continuity) is presented in
Theorem~\ref{DeltaDefVerified} below.  The proof is partly based on our
improvement of the method of complex interpolation explained below.

\subsection*{Complex method of interpolation}

We shall now briefly recall the complex method of interpolation.
For a compatible pair of Banach spaces~$\left( A_0, A_1 \right)$,
and~$0 < \theta < 1$, the complex interpolation Banach
space~$A_\theta = (A_0, A_1)_\theta$ is defined as follows (see
e.g. \cite[Section 4.1]{KPS}): $$ A_\theta := \left\{x \in A_0 +
  A_1:\ \ \exists f \in \mathcal{F}(A_0, A_1)\ \ \text{such that}\ \
  x = f(\theta) \right\}. $$ Here the class~$\mathcal{F} (A_0,
A_1)$ consists of all  bounded and continuous functions~$f : \bar
S \mapsto A_0 + A_1$ defined on the closed strip $$\bar S :=
\left\{z \in\Cx:\ \ 0
  \leq \Re z \leq 1 \right\} $$  such that $f$ is analytic on the open strip
$S := \left\{z \in\Cx:\ \ 0  < \Re z < 1 \right\}$ and such that
$t\to f(j + it) \in A_j$, $j=0,1$ are continuous functions on the
real line.
%, which tend to $0$ as $|t|\to \infty$.
We provide $\mathcal{F}(A_0, A_1)$ with the norm
$$
\|f\|_{\mathcal{F}(A_0, A_1)}:=\max_{j=0,1}\{c_{0}(f), c_{1}(f)\},
$$
where $c_j(f):=\sup_{t \in \Rl} \left\| f(j + it) \right\|_{A_j},$
$j=0,1.$

Setting
$$\left\| x \right\|_{A_\theta}:=\inf\{\,\|f\|_{\mathcal{F}(A_0,
A_1)}\ :\ f(\theta)=x,\ f\in \mathcal{F}(A_0, A_1)\}
$$
we obtain a Banach space $(A_\theta, \|\cdot\|_{A_\theta})$. It is
well known that $\left\| x \right\|_{A_\theta}\leq c_{0}^{1 -
\theta}(f)\, c_{1}^\theta(f)$, where $f(\theta)=x$, $f\in
\mathcal{F}(A_0, A_1)$.

\begin{lemma}
  \label{CIlemmaReplacement}
  Let~$F_z$ be the multilinear operator
$$ F_z:\underbrace{\cS^\infty
    \times \ldots \times \cS^\infty}_{\text{$m$-times}}\mapsto \cS^\infty,\ \ \forall z \in \bar S,$$
%where
%$$ S = \left\{z \in \mathbb{C},\ \ 0 < \Re z < 1 \right\}.$$
such that $z \mapsto F_z$ is  analytic in $S.$  If the constants
$$c_j = \sup_{t\in \Rl} \left\|F_{j + it} \right\|_{\tilde q^{(j)} \mapsto r_j},\mbox{ where }\tilde q^{(j)} = \left( q^{(j)}_1, \ldots, q^{(j)}_m \right)\mbox{ and }j = 0, 1$$
are finite, then 
  \begin{multline*}
    \left\| F_{\theta}
    \right\|_{\tilde q \mapsto r} \leq c_0^{1-\theta} c_1^\theta,\ \
    \text{where}\ \ \tilde q = (q_1, \ldots, q_m) \\ \text{and}\ \
    \frac 1 {q_k} = \frac {1-\theta}{q_k^{(0)}} + \frac
    \theta{q_k^{(1)}},\ \ k = 1, \ldots, m \\ \text{and}\ \ \frac 1r =
    \frac {1-\theta}{r_0} + \frac \theta{r_1}. 
  \end{multline*}
\end{lemma}
\begin{proof} Fix $\varepsilon>0.$ For every $1\leq k\leq m,$ there exists a function $g_k\in\mathcal{F}(\cS^{q_k^{(0)}},\cS^{q_k^{(1)}})$ such that
\begin{multline*}
  g_k(\theta)=V_k \ \ \text{and}\\ { { \left\|
        V_k\right\|_{q_k} \leq}} { \sup_{t \in \Rl}}\,
  { \max}
  \left\{\|g_k(it)\|_{{q_k^{(0)}}},\|g_k(1+it)\|_{{q_k^{(1)}}}
  \right\}\leq(1+\varepsilon)\|V_k\|_{q_k}.
\end{multline*}
Define an analytic
function $h$ in the strip $S$ by setting
$$h(z)=F_z(g_1(z),\cdots,g_m(z)).$$
By the assumption,
\begin{multline*}
  \|h(it)\|_{{r_0}}\leq\|F_{it}\|_{ \tilde q^{(0)}
    \mapsto 
    r_0}\|g_1(it)\|_{{q_1^{(0)}}}\cdots\|g_m(it)\|_{{q_m^{(0)}}}
  \\ \leq(1+\varepsilon)^m
  c_0\|V_1\|_{{q_1}}\cdots\|V_m\|_{{q_m}}.
\end{multline*}
Similarly,
$$\|h(1+it)\|_{{r_0}}\leq(1+\varepsilon)^m c_1\|V_1\|_{{q_1}}\cdots\|V_m\|_{{q_m}}.$$
It follows from the definition of complex interpolation method combined with the fact that $(\cS^{r_0},\cS^{r_1})_{\theta}=\cS^r,$ that
$$\|F_{\theta}(V_1,\cdots,V_m)\|_{r}\leq\|h(it)\|_{{r_0}}^{\theta}\|h(1+it)\|_{{r_1}}^{1-\theta}.$$
Since $\varepsilon$ is arbitrarily small, it follows that
$$\|F_{\theta}(V_1,\cdots,V_m)\|_{r}\leq c_0^{\theta}c_1^{1-\theta}\|V_1\|_{{q_1}}\cdots\|V_m\|_{{q_m}}.$$
%This concludes the proof.
\end{proof}

\subsection*{The functionals~$\delta^{[k]}_H$ are well-defined}

The following theorem is the key to showing that the
functionals~$\delta^{[k]}_H$'s are well-defined and continuous.

\begin{theorem}
  \label{DeltaDefVerified}
  If
%\/~$\tilde H := (\underbrace{H, \ldots,      H}_{\text{$k$-times}})$ and
~$H \in
  \cS^p$, then the operator~$T_{\tilde f^{[k]}_p}^{ H}$ maps
  $\underbrace{\cS^p \times \ldots \times \cS^p}_{\text{$k-1$-times}}
  \mapsto \cS^{p'}$, for every integral~$2 \leq k < p$.  Moreover, $$
  \left\| T_{\tilde f^{[k]}_p}^{ H} (V_1, \ldots, V_{k-1})
  \right\|_{p'} \leq \const \left\|  H \right\|_p^{p - k}\,
  \left\| V_1 \right\|_p \cdot \ldots \cdot \left\| V_{k-1}
  \right\|_{p}, $$ where $\frac{1}{p}+\frac{1}{p'}=1$.
\end{theorem}

\begin{proof}[Proof of Theorem~\ref{DeltaDefVerified}]
  We fix~$V_1, \ldots, V_{k-1} \in \cS^p$ and assume that~$\left\| V_j
  \right\|_p = 1$, $j = 1, \ldots, k-1$.  By Lemma \ref{HolderToBesov} applied to the smoothed function $f_p$, 
we have $f_p\in \tilde B^k_{\infty, 1}$ for every positive integral $k<p$.  Hence the function $h=f^{(k)}_p$ belongs to 
$\tilde B^0_{\infty, 1}$ and since $\tilde f^{[k]}_p$ is a $(k-1)$-th order polynomial integral momentum
associated with the function $h=f^{(k)}_p$ (see (\ref{DDtilde})) we infer from Theorem \ref{PIMpolyEst} that 
$\tilde f^{[k]}_p\in \cC_{k-1}$. Now, by Lemma \ref{lemma1}, we have $\hat T_{\tilde f^{[k]}_p}^{
H}=T_{\tilde f^{[k]}_p}^{ H}$ and applying Theorem~\ref{PSStheorem} we obtain
$$\| T_{\tilde f^{[k]}_p}^{ H}\|_{\tilde
p\to \frac p{k-1}}\le\const\,\|f^{(k)}_p\|_{\infty},$$
where $\tilde p=(\underbrace {p, p,\dots, p}_{\text{$k-1$-times}})$.
In particular, $$ T_{\tilde
f_p^{[k]}}^{ H}\left(
    \tilde V \right) \in \cS^{\frac p{k-1}},\ \ \text{where}\ \ \tilde
  V = \left( V_1, \ldots, V_{k-1} \right). $$
However, the estimate above is weaker than the claim of Theorem \ref{DeltaDefVerified}. To achieve the claim, we need a rather delicate application of the complex interpolation
  Lemma~\ref{CIlemmaReplacement}.
 For the rest of the proof, we fix an integral~$n \geq 0$ such
  that $$ 2n < p -k \leq 2n + 2. $$  In order to use  Lemma~\ref{CIlemmaReplacement}, we will
  construct a family of analytic operator valued functions $$
  z \in \Cx,\ \epsilon > 0 \mapsto F_{z, \epsilon} :
  \underbrace{\cS^\infty \times \ldots \times \cS^\infty}_{\text{$k-1$
      times}} \mapsto \cS^\infty, $$ such that $$ F_{j + it,
    \epsilon}: \underbrace{\cS^{q_j} \times \ldots \times
    \cS^{q_j}}_{\text{$k-1$ times}} \mapsto \cS^{r_j},\ \ j = 0, 1, $$
  where the exponents~$q_j$ and~$r_j$ are given by 
  \begin{multline*}
    \frac 1{q_0} =
    \frac {p-2n}{kp} \ \ \text{and}\ \ \frac 1{q_1} = \frac {p - 2n -
      2}{kp},\\ \frac 1{r_0} = \frac {2n}p + \frac {k-1}{q_0} \ \
    \text{and}\ \ \frac 1{r_1} = \frac {2n+2}p + \frac {k-1}{q_1}. 
  \end{multline*}
  We observe right away that due to the assumption ~$2 \leq k < p$ and the choice
  of~$n$, the indices~$r_j, q_j$ are non-trivial, that is $$ 1 < r_j,
  q_j < \infty, \ \ j = 0, 1. $$ In addition, the family~$F_{z,
    \epsilon}$  will have also satisfied the boundary estimates 
  \begin{equation}
    \label{DeltaDefVerifiedBoundaryEst}
    \left\|
    F_{j + it, \epsilon} \left( \tilde V \right) \right\|_{r_j} \leq
  \const \, \left\| H \right\|_p^{2n + 2j}\, \left\| V_1\right\|_{q_j}
  \cdot \ldots \cdot \left\| V_{k-1} \right\|_{q_j}, 
\end{equation}
with the constant in  \eqref{DeltaDefVerifiedBoundaryEst} being independent of~$\epsilon > 0$ and
  such that $$ F_{\theta, \epsilon} \left( \tilde V \right) = T^H_{\tilde
    f^{[k]}_{p + \epsilon}} \left( \tilde V \right),\ \ \text{where}\
  \ \theta = \frac {p -k}2 - n. $$

  Given the family~$F_{z, \epsilon}$ as above and using
  Lemma~\ref{CIlemmaReplacement}, we readily arrive at the estimate $$
  \left\| F_{\theta, \epsilon} \left( \tilde V \right) \right\|_{p'} =
  \left\| T^H_{\tilde f^{[k]}_{p + \epsilon}} \left( \tilde V \right)
  \right\|_{p'} \leq \const\, \left\| H \right\|_p^{p - k}\, \left\|
    V_1\right\|_p\ \cdot \ldots \cdot \left\| V_{k-1} \right\|_p, $$
where the constant is independent of~$\epsilon > 0$.  The claim of the theorem
   $$ \left\| T^H_{\tilde f^{[k]}_{p}} \left( \tilde V \right)
  \right\|_{p'} \leq \const\, \left\| H \right\|_p^{p - k}\, \left\|
    V_1\right\|_p\ \cdot \ldots \cdot \left\| V_{k-1} \right\|_p $$
now follows from Lemma~\ref{DomConvLemma}, which is applicable due to pointwise convergence $$
  \lim_{\epsilon \rightarrow 0} f_{p+\epsilon}^{(k)} (x) =
  f^{(k)}_p(x),\ \ x \in \Rl, $$ of compactly supported continuous functions. 

We now focus on the construction of the family~$\left\{F_{z,  \epsilon} \right\}_{\epsilon > 0}$.  The construction is based
  on the following auxiliary lemma.

  \begin{lemma}
    \label{DeltaDefVerifiedLemmaAux}
    Let~ $f_z(x) := \left[ f_1(x) \right]^z$, $z\in \mathbb{C}$ be the
    analytic continuation of~the mapping $p \rightarrow f_p$
    to~$\mathbb{C}$, and let ~$$ z \mapsto \tilde F_{z}:
    \underbrace{\cS^\infty \times \ldots \times
      \cS^\infty}_{\text{$k-1$ times}} \mapsto \cS^\infty $$ be the
    analytic (in $\mathbb{C}$) operator valued function given by $$
    \tilde F_{z} \left( \tilde V \right) = T^H_{\tilde f^{[k]}_z}
    \left( \tilde V \right), \ \ { \tilde V = (V_1, \ldots ,
    V_{k-1})}. $$ Let~$m \geq 0$ be an integer such that~$\Re z > 2m +
    k$.  If\/~$1 < r, q < \infty$ are such that $$ \frac 1r = \frac
    {2m}p + \frac {k-1}q, $$ then, we have the following estimate
\begin{equation}\label{n1}
\left\|
      \tilde F_z \left( \tilde V \right) \right\|_r \leq \const\,
    \left( 1 + \left| \Im z \right| \right)^k\, \left\| H
    \right\|_p^{2m} \, \left\| V_1 \right\|_q \cdot \ldots \cdot
    \left\| V_{k-1} \right\|_q
\end{equation} with constant being independent of~$z$.
  \end{lemma}

  The proof of the lemma will follow momentarily.  However, we shall
  first finish the proof of the theorem.  Given the lemma above it is
  now straightforward.  Indeed, we choose the family~$\left\{F_{z,
      \epsilon} \right\}_{\epsilon > 0}$ as follows $$ F_{z, \epsilon}
  = e^{z^2 - \theta^2}\, \tilde F_{2z + 2n + k + \epsilon}. $$
  Clearly, $$ F_{\theta, \epsilon} = \tilde F_{p + \epsilon} =
  T^H_{\tilde f^{[k]}_{p + \epsilon}}. $$ Also, the boundary
  estimates~(\ref{DeltaDefVerifiedBoundaryEst}), both follow from the
  lemma with~$r = r_j$, $q=q_j$, $m = n + j$, $z=2j+2n+k+\epsilon$, $j=0,1$.  Observe that the
  polynomial growth with respect to~$\left| \Im z \right|$ in the
  lemma is controlled by the exponential decay of the function~$e^{z^2 - \theta^2}$ on the boundary of the strip~$S$.  Thus, the theorem
  is completely proved.
\end{proof}

\begin{proof}[Proof of Lemma~\ref{DeltaDefVerifiedLemmaAux}]
 We write the $k$-th derivative as 
    \begin{multline*}
      f^{(k)}_z (x)\, {=\frac{z(z-1)(z-2)\ldots (z-k+1)}{( 1 + \left|
            \Im z \right| )^k}(\sgn x)^k( 1 + \left| \Im z
        \right|)^k|x|^{z-k}}\\ = { w(z)}\, \left( 1 + \left| \Im z
        \right| \right)^k\, \left| x \right|^{z - k}, \ \ \left| x
      \right| \leq 2,\ \ \text{where} \\ { \sup_{z \in \bar
        S} \left| w(z) \right|
      \leq \const,\ \ w(z) {=(\sgn x)^k\frac{z(z-1)(z-2)\ldots (z-k+1)}{( 1 +
          \left| \Im z \right| )^k}}}. \end{multline*}  
The derivative above is continuous,
    since~$\Re z > k$.  Moreover, by Lemma~\ref{HolderToBesov}, the
    function $h=f^{(k)}_z$ belongs to $\tilde B^0_{\infty, 1}$.

    We assume for simplicity that $$ \left\| V_1 \right\|_q = \ldots =
    \left\| V_{k-1} \right\|_q = 1. $$ We then have to show that
    \begin{equation}
      \label{DeltaDefVerifiedWeakEst}
      \left\| T_{\tilde f^{[k]}_{z}}^H \left(\tilde V
        \right) \right\|_r \leq \,
      \const \left( 1 +
        \left| \Im z \right|\right)^k\, \left\| H \right\|^{2m}_p.
    \end{equation}
Recall (see \eqref{DDintRep}) that~$\tilde f^{[k]}_{z}$ is the polynomial
    integral momentum of order~$k-1$ associated with the polynomial~$Q
    = 1$ and the function~$h = f^{(k)}_{z}$. Using the polynomial
    expansion
  \begin{multline*}
    \left| s_0 x_0 + \ldots + s_{k-1} x_{k-1} \right|^{2m} \\ =
    \sum_{m_0 + \ldots + m_{k-1} = 2m} C_{m_0, \ldots, m_{k-1}} \,
    (s_0 x_0)^{m_0} \cdot \ldots \cdot (s_{k-1} x_{k-1})^{m_{k-1}} \\
    \text{where}\ \ C_{m_0, \ldots, m_{k-1}} = \frac {\left( 2m
      \right)!}{m_0! \cdot \ldots \cdot m_{k-1}!},
  \end{multline*}
we represent the momentum~$\tilde f^{[k]}_{z}$ via \eqref{DDintRep} as follows 
  \begin{multline*}
    \tilde f^{[k]}_{z} {(x_0,\ldots ,x_{k-1}) :=
      \int_{S_{k-1}} f^{(k)}_{z}\left(
        s_0x_0+\ldots +s_{k-1}x_{k-1} \right) \, d\sigma_{k-1}} \\ ={ w(z)\,}
    \left( 1 + \left| \Im z \right| \right)^k\,\int_{S_{k-1}} \left|
      s_0 x_0 + \ldots + s_{k-1} x_{k-1} \right|^{z - k} \,
    d\sigma_{k-1} \\ = { w(z)\,} \left( 1 + \left| \Im z \right|
    \right)^k\,\int_{S_{k-1}} \left| s_0 x_0 + \ldots + s_{k-1}
      x_{k-1} \right|^{z - k - 2m + 2m} \, d\sigma_{k-1} \\ = { w(z)\,}
    \left( 1 + \left| \Im z \right| \right)^k\,\int_{S_{k-1}} \tilde
    h_{z}(s_0x_0+\ldots +s_{k-1}x_{k-1})\, \left| s_0 x_0 + \ldots +
      s_{k-1} x_{k-1} \right|^{2m}\, d\sigma_{k-1} \\ ={ w(z)\,} \left( 1
      + \left| \Im z \right| \right)^k\,\int_{S_{k-1}} \tilde
    h_{z}(s_0x_0+\ldots +s_{k-1}x_{k-1}) \\ \times\sum_{m_0 + \ldots +
      m_{k-1} = 2m} C_{m_0, \ldots, m_{k-1}} \, (s_0 x_0)^{m_0} \cdot
    \ldots \cdot (s_{k-1} x_{k-1})^{m_{k-1}} \, d\sigma_{k-1} \\
    ={ w(z)\,} \left( 1 + \left| \Im z \right| \right)^k\,\sum_{m_0 +
      \ldots + m_{k-1} = 2m} C_{m_0, \ldots, m_{k-1}} \,x_0^{m_0}
    \cdot \ldots \cdot x_{k-1}^{m_{k-1}} \\ \times\int_{S_{k-1}}
    s_0^{m_0} \cdot \ldots \cdot s_{k-1}^{m_{k-1}}\, \tilde
    h_{z}(s_0x_0+\ldots +s_{k-1}x_{k-1}) \, d\sigma_{k-1} \\ ={ w(z)\,}
    \left( 1 + \left| \Im z \right| \right)^k\,\sum_{m_0 + \ldots +
      m_{k-1} = 2m} C_{m_0, \ldots, m_{k-1}} \,x_0^{m_0} \cdot \ldots
    \cdot x_{k-1}^{m_{k-1}}\,\phi_{z, m_0, \ldots, m_{k-1}} \left(
      x_0, \ldots, x_{k-1}\right)\\ ={ w(z)\,} \left( 1 + \left| \Im z
      \right| \right)^k \sum_{m_0 + \ldots + m_{k-1} = 2m} C_{m_0,
      \ldots, m_{k-1}}\, f_{z, m_0, \ldots, m_{k-1}}{(x_0,
      \ldots, x_{k-1})},
  \end{multline*}
  where $$ f_{z, m_0, \ldots, m_{k-1}} (x_0, \ldots, x_{k-1}) =
  x_0^{m_0} \cdot \ldots \cdot x_{k-1}^{m_{k-1}} \cdot \phi_{z, m_0,
    \ldots, m_{k-1}} \left( x_0, \ldots, x_{k-1} \right) $$ and where
  the {$(k-1)$-th} polynomial integral
  momentum~$\phi_{{z,}m_0, \ldots, m_{k-1}}$ is associated
  with the function $$ \tilde h_{z}(x) = \left| x \right|^{z-k-2m} $$
  and the polynomial $$ Q_{m_0, \ldots, m_{k-1}} (s_1, \ldots,
  s_{k-1}) = s_0^{m_0} \cdot \ldots \cdot s_{k-2}^{m_{k-2}} \cdot
  s_{k-1}^{m_{k-1}}. $$ Thus, to see~(\ref{DeltaDefVerifiedWeakEst})
  it is sufficient to estimate the individual summand in the multiple
  operator integral $$T_{\tilde f^{[k]}_{z}}^{ H}(\tilde
  V)= { w(z)\,} \left( 1 + \left| \Im z \right| \right)^k \sum_{m_0 +
    \ldots + m_{k-1} = 2m} C_{m_0, \ldots, m_{k-1}}\, T^{
    H}_{f_{z, m_0, \ldots, m_{k-1}}}(\tilde V).$$ Looking at the
  integrals associated with individual functions~$f_{z, m_0, \ldots,
    m_{k-1}}$ and appealing to ~(\ref{TphiHomomorphismI}), we obtain
\begin{equation}\label{n2}T_{f_{z,
      m_0, \ldots, m_{k-1}}}^{ H} (V_1, \ldots, V_{k-1}) =
  T_{\phi_{z, m_0, \ldots, m_{k-1}}}^{ H} \left( H^{m_0} V_1
    H^{m_1}, V_2 H^{m_2}, \ldots, V_{k-1} H^{m_{k-1}}
  \right).\end{equation} 

Recall that the function $\tilde h_{z}$ is compactly supported and that $\|\tilde h_{z}\|_\infty$ is uniformly bounded with respect to $0\leq z\leq 2$. Hence applying 
firstly Theorem~\ref{PSStheorem} to the right hand side of \eqref{n2} and then using H\"older inequality we arrive at
  and
  \begin{multline*}
    \left\| T_{f_{z, m_0, \ldots, m_{k-1}}}^{H}(V_1, \ldots, V_{k-1})
    \right\|_r \\ \leq \const\, \left\| H^{m_0} V_1 H^{m_1}
    \right\|_{\alpha_1}\, \left\| V_2 H^{m_2} \right\|_{\alpha_2}
    \cdot \ldots \cdot \left\| V_{k-1} H^{m_{k-1}}
    \right\|_{\alpha_{k-1}} \\ \leq \const \left\| H\right\|_p^{m_0}
    \left\| V_1 \right\|_q\, \left\| H\right\|_p^{m_1}\, \left\| V_2
    \right\|_q \, \left\| H\right\|_p^{m_2}\cdot \ldots \cdot \left\|
      V_{k-1} \right\|_q\,
    \left\| H\right\|_p^{m_{k-1}} \\
    \leq \const\, \left\| H \right\|_p^{2m}\, \left\| V_1 \right\|_q
    \cdot \ldots \cdot \left\| V_{k-1} \right\|_q \\ = \const\,
    \left\|
      H \right\|_p^{2m}  \ \ \text{where}\\
    \frac 1{\alpha_1} = \frac {m_0 + m_1}p + \frac 1{q},\ \ \frac
    1{\alpha_j} = \frac {m_j}p + \frac 1q,\ \ j = 2, \ldots, k-1.
  \end{multline*}
  Thus, the estimate~(\ref{DeltaDefVerifiedWeakEst}) is
  shown.  The proof of the lemma is finished.
\end{proof}

\begin{remark} The assumption $\Re z> 2m+k$ is used because the RHS of \eqref{n2} does not make sense when $\Re z=2m+k$. 
However, the LHS in \eqref{n1} does make sense even when $\Re z=2m+k,$ unless $m=0$. Consequently, 
the assertion of Lemma \ref{DeltaDefVerifiedLemmaAux} holds for $\Re z\geq 2m+k$ when $m>0$.
\end{remark}

\subsection*{Taylor expansion for~$t \mapsto \tr \left( f_p(H_t)
  \right)$}

Here, we deal with the final step of the proof of Theorem~\ref{TaylorExpansion}.  Let~$1 < p < \infty$. Fix self-adjoint elements~$H_1,
H_0 \in \cS^p$ such that $\|H_j\|_p\leq 1$, $j=0,1$ and set~$H_t := (1 - t)\, H_0 + t H_1$,  $V:=H_1-H_0$. 
%Let also~$\tilde H_k = \Bigl( \underbrace{H_0, \ldots,H_0}_{\text{$k-1$-times}} \Bigr)$. 
We begin our discussion of Taylor expansion of ~$t \mapsto \tr \left( f_p(H_t) \right)$ with the simplest case, when $m=1$. Firstly, by the fundamental theorem of
the calculus, we write $$ \tr \left( f_p(H_{ 1}) \right) -\tr \left( f_p(H_0) \right) = \int_0^1 \frac {d}{dt} \tr \left(f_p(H_t) \right)\, dt. $$
Now, we use well-known formulae (see e.g. \cite{BiSo3} or \cite[Corollary 6.8]{PS-DiffP} together with \cite[Lemma 20]{CaPoSu}) and rewrite the preceding formula as
 $$ \tr \left( f_p(H_{ 1}) \right) -\tr \left( f_p(H_0) \right) = \int_0^1 \tr \left( V\, f'_p(H_t)\right)\, dt. $$
Now, we claim that the following formula holds for all $m \in \N:\ 1 < m < p$.
\begin{multline}
  \label{TaylorExpansionFinal}
  \tr \left( f_p\left( H_1 \right) \right) = \tr(f_p(H_0)) + \sum_{k = 1}^{m-1} \frac 1{k} \, \tr \Bigl( V\, T^{H_{0}}_{\tilde
    f^{[k]}_p}\Bigl(\underbrace{V, \ldots,
    V}_{\text{$k-1$-times}}\Bigr) \Bigr)\\ + \int_0^1 t^{m-1}\, \tr
  \Bigl( V\, T^{H_t,  H_0}_{\tilde
    f^{[m]}_p}\Bigl(\underbrace{V, \ldots,
    V}_{\text{$m-1$-times}}\Bigr) \Bigr)\, dt.
\end{multline}
In writing the multiple operator integral in the first line of ~(\ref{TaylorExpansionFinal}) we used convenion \eqref{abbreviation}, 
whereas in the second line we used a similar convention by replacing $ \Bigl(H_t,  \underbrace{H_0, \ldots,H_0}_{\text{$m-1$-times}} \Bigr)$ with just $(H_t,  H_0)$.

We prove~(\ref{TaylorExpansionFinal}) by the method of mathematical
induction.  The induction step is justified as follows. Assuming that
~(\ref{TaylorExpansionFinal}) holds for~$m -1$, we add and
subtract $$ \frac 1{m-1}\, \tr \Bigl( V\,
T^{{  H_{0}}}_{\tilde
f^{[m-1]}_p}\Bigl(\underbrace{V, \ldots,
  V}_{\text{$m-2$-times}}\Bigr) \Bigr),\ \ \left[ \pm \tr \left(V\, \tilde f'_p(H_0) \right)\ \text{if~$m = 2$} \right] $$ We then have
\begin{multline*}
  \tr \left( f_p\left( H_1 \right) \right) =
{ \tr(}f_p(H_0){ )} + \sum_{k =
    1}^{m-2} \frac 1{k}\, \tr \Bigl( V\, T^{{ 
H_{0}}}_{\tilde
    f^{[k]}_p}\Bigl(\underbrace{V, \ldots,
    V}_{\text{$k-1$-times}}\Bigr) \Bigr) \\ + \frac 1{m-1} \, \tr
  \Bigl( V\, T^{{  H_{0}}}_{\tilde
f^{[m-1]}_p}\Bigl(\underbrace{V, \ldots,
    V}_{\text{$m-2$-times}}\Bigr) \Bigr) \\ + \int_0^1 t^{m-2}\, \tr
  \Bigl( V\, T^{H_t,  H_{0}}_{\tilde
    f^{[m-1]}_p}\Bigl(\underbrace{V, \ldots,
    V}_{\text{$m-2$-times}}\Bigr) \Bigr) \, dt \\ - \frac 1 {m-1}\, \tr
  \Bigl( V\, T^{{  H_{0}}}_{\tilde
f^{[m-1]}_p}\Bigl(\underbrace{V, \ldots,
    V}_{\text{$m-2$-times}}\Bigr) \Bigr).
\end{multline*}
For the last two terms we now observe that
\begin{multline*}
  \int_0^1 t^{m-2}\, \tr \Bigl( V\, T^{H_t,  H_{0}}_{\tilde
    f^{[m-1]}_p}\Bigl(\underbrace{V, \ldots,
    V}_{\text{$m-2$-times}}\Bigr) \Bigr) - \frac 1{m-1}\, \tr \Bigl(
  V\, T^{{  H_{0}}}_{\tilde f^{[m-1]}_p}\Bigl(\underbrace{V,
\ldots,
    V}_{\text{$m-2$-times}}\Bigr) \Bigr) \\ = \int_0^1 t^{m-2}\,
  \left[ \tr \Bigl( V\, T^{H_t,  H_{0}}_{\tilde
      f^{[m-1]}_p}\Bigl(\underbrace{V, \ldots,
      V}_{\text{$m-2$-times}}\Bigr) \Bigr) - \tr \Bigl( V\,
    T^{{  H_{0}}}_{\tilde f^{[m-1]}_p}\Bigl(\underbrace{V,
\ldots,    V}_{\text{$m-2$-times}}\Bigr) \Bigr) \right]\, dt
\end{multline*}
Finally, by Theorem~\ref{PertubationFormula} (with $\phi=\tilde f^{[m-1]}_p$, $\psi=\tilde f^{[m]}_p$ and $h= f^{(m-1)}_p$, which belongs to the space 
$\tilde B^1_{\infty, 1}$ by Lemma \ref{HolderToBesov}), we see that 
$${ T^{H_t,
   H_{0}}_{\tilde f^{[m-1]}_p}\Bigl(\underbrace{V, \ldots,
  V}_{\text{$m-2$-times}}\Bigr) - T^{H_0}_{\tilde
  f^{[m-1]}_p}\Bigl(\underbrace{V, \ldots,
  V}_{\text{$m-2$-times}}\Bigr)\stackrel{Th.\ref{PertubationFormula}}=T^{H_t,H_0}_{\tilde
  f^{[m]}_p}\Bigl(H_t-H_0,\underbrace{V, \ldots,
  V}_{\text{$m-2$-times}}\Bigr)}$$$${ =T^{H_t,  H_{0}}_{\tilde
  f^{[m]}_p}\Bigl(t(H_1-H_0),\underbrace{V, \ldots,
  V}_{\text{$m-2$-times}}\Bigr)}= t T^{H_t,  H_{0}}_{\tilde
  f^{[m]}_p}\Bigl(\underbrace{V, \ldots,
  V}_{\text{$m-1$-times}}\Bigr). $$
This completes the proof of formula~(\ref{TaylorExpansionFinal}). Now, we are in a position to prove 
the expansion ~(\ref{eq:TaylorExpansion1}).
By adding and
subtracting to ~(\ref{TaylorExpansionFinal}) $$ \frac 1{m}\, \tr \Bigl( V\,
T^{{  H_{0}}}_{\tilde
  f^{[m]}_p}\Bigl(\underbrace{V, \ldots, V}_{\text{$m-1$-times}}\Bigr)
\Bigr), $$ we obtain that
\begin{multline*}
  \label{TaylorExpansionFinal}
  \tr \left( f_p\left( H_1 \right) \right) =
{ \tr(}f_p(H_0){ )} + \sum_{k =
    1}^{m-1} \frac 1{k} \, \tr \Bigl( V\, T^{{  H_{0}}}_{\tilde
    f^{[k]}_p}\Bigl(\underbrace{V, \ldots,
    V}_{\text{$k-1$-times}}\Bigr) \Bigr)\\ + \int_0^1 t^{m-1}\, \left[
    \tr \Bigl( V\, T^{H_t,  H_0}_{\tilde
      f^{[m]}_p}\Bigl(\underbrace{V, \ldots,
      V}_{\text{$m-1$-times}}\Bigr) \Bigr) - \tr \Bigl( V\,
    T^{{  H_{0}}}_{\tilde f^{[m]}_p}\Bigl(\underbrace{V,
\ldots,
      V}_{\text{$m-1$-times}}\Bigr) \Bigr) \right]\, dt.
\end{multline*}
Observe that the first line above, given the definition of the
functionals~$\delta^{[k]}_H$, is the complete left hand side
of~(\ref{eq:TaylorExpansion1}).  In other words, we have
\begin{multline*}
  \left\| H_{ 0} + V \right\|_p^p - \left\| H_{ 0}
\right\|_p^p -
  \sum_{k = 1}^m \dfun{k} \Bigl(\underbrace{V, \ldots,
    V}_{\text{$k$-times}}\Bigr) \\ = \int_0^1 t^{m-1}\, \left[ \tr \Bigl(
    V\, T^{H_t, \tilde H_m}_{\tilde f^{[m]}_p}\Bigl(\underbrace{V,
      \ldots, V}_{\text{$m-1$-times}}\Bigr) \Bigr) - \tr \Bigl( V\,
    T^{{ \tilde H_{m+1}}}_{\tilde f^{[m]}_p}\Bigl(\underbrace{V,
\ldots,
      V}_{\text{$m-1$-times}}\Bigr) \Bigr) \right]\, dt.
\end{multline*}
Setting in Theorem~\ref{HolderInFull}  (and  Remark~\ref{HolderInFullAlphaOne}) $A=H_t,B=H_0, V_1=\cdots=V_{m-1}=V$, $p_0=p_1=\cdots=p_{m-1}=p$, 
$\alpha=p-m$ and applying that theorem with $m-1$ instead of $m$,  
we have $\alpha/p+(m-1)/p=(p-m)/p+(m-1)/p=1/p'$ where $\frac 1p+\frac{1}{p'}=1$
and
 $$
\left\|T^{H_t,  H_0}_{\tilde f^{[m]}_p}\Bigl(\underbrace{V, \ldots,
  V}_{\text{$m-1$-times}}\Bigr) - T^{{  H_{0}}}_{\tilde
  f^{[m]}_p}\Bigl(\underbrace{V, \ldots, V}_{\text{$m-1$-times}}\Bigr)\right\|_{p'}
= O\left( \left\| V \right\|_p^{p-1} \right). $$
 
Thus, the expansion~(\ref{eq:TaylorExpansion1}) follows immediately via H\" older inequality.
Theorem~\ref{TaylorExpansion} is completely proved.

\section{Concluding remarks}

As we noted in the Introduction, our methods are also suitable to resolve a similar problem concerning 
differentiability properties of general non-commutative $L_p$-spaces associated with general semifinite von Neumann algebras $\mathcal M$ stated in \cite{PX}. 
In this paper we demonstrate such a resolution for the special case when $\mathcal{M}$ is an arbitrary type $I$ von Neumann algebra acting on a 
separable Hilbert space. Using well-known structural results for such algebras, it is easy to see  that it is suffcient to deal with 
$L_p$-spaces associated with the von Neumann tensor product 
$L_\infty(0,1)\bar \otimes B(H)$, or equivalently, it is sufficient to deal with Lebesgue-Bochner spaces 
$L_p(S^p):= L_p([0,1], S^p)$ (see e.g. \cite{Suk} and references therein). The case $1\leq p\leq 2$ is easy and in fact has been dealt with in full generality in 
\cite[Lemma 4.1]{DD}. In order to deal with the case $2\leq p<\infty$ and thus, with  higher order derivatives, it
 is convenient to cite the following result from \cite{LS-1974}.

\begin{theorem} \cite[Theorem 3.5]{LS-1974} Let $E$ be a Banach space, $(T, \Sigma, \mu)$ a measure space, and $k$ a positive
integer with $p > k$. If the norm $\|\cdot\|: E\mapsto \Rl$ is $k$-times continuously differentiable
away from zero and the $k$-th derivative of the norm in $E$ is uniformly bounded on the
unit sphere in $E$, then the norm $\|\cdot\|: L_p(E,\mu) \mapsto \Rl$ is $k$-times continuously
differentiable away from zero.
\end{theorem}

In view of this result, we shall obtain the result similar to Theorem \ref{TJ} as soon as we verify the assumptions of the preceding theorem concerning the norm 
$\|\cdot\|: S^p\mapsto \Rl$. We shall verify these assumptions on the (open) unit ball $S^p_1$ of $S^p$. 
To this end, we firstly need to verify the the mapping $\|\cdot\|_p: S^p\mapsto \Rl$ is $k$-times continuously differentiable
that is (see  \cite[p.232]{LS-1974}) it is  $k$-times differentiable and the $k$-th derivative $\delta^{(k)}$ is continuous on $S^p_1$. 
Secondly, we need to ascertain that the derivative $\delta^{(k)}$  is uniformly bounded on $S^p_1$. Now, the existence of the derivative 
$\delta^{(k)}$ is of course our main result Theorem \ref{TaylorExpansion}. The continuity  of $\delta^{(k)}$ on the unit sphere of $S^p$ follows from the definitions 
\eqref{deltaFnDef_square} and \eqref{deltaFnDef} together with the estimate obtained in Theorem \ref{DeltaDefVerified} and H\" older inequality. 
Finally, the same estimate also yields uniform boundedness of $\delta^{(k)}$ on $S^p_1$. This completes the proof of analogue of Theorem \ref{TJ} for the space $L_p(S^p)$. 
The general case of an arbitrary semifinite von Neumann algebra $\mathcal M$ of type $II$ depends on substantial technical preparations 
needed to extend definitions \eqref{deltaFnDef_square} and \eqref{deltaFnDef} to the setting of unbounded operators from $L_p(\mathcal M)$ 
 and will be dealt separately.

% {\color{red} I am not sure about the display above, may be it is
% like that
% $$
% \Big\|T^{H_t, \tilde H_m}_{\tilde f^{[m]}_p}\Bigl(\underbrace{V,
% \ldots,
%   V}_{\text{$m-1$-times}}\Bigr) - T^{\tilde H_{m+1}}_{\tilde
%   f^{[m]}_p}\Bigl(\underbrace{V, \ldots,
%   V}_{\text{$m-1$-times}}\Bigr)\Big\|_{\frac p{p-1}}
% $$$$=\Big\|T^{H_t, \tilde H_m}_{\tilde f^{[m]}_p}\Bigl(\underbrace{V,
% \ldots,
%   V}_{\text{$m-1$-times}}\Bigr) - T^{H_0,\tilde H_{m}}_{\tilde
%   f^{[m]}_p}\Bigl(\underbrace{V, \ldots,
%   V}_{\text{$m-1$-times}}\Bigr)\Big\|_{\frac
% p{p-1}}\stackrel{T\ref{HolderInFull}}\le
%   \const\|f^{(m)}_p\|_{\Lambda_{\alpha}}\|H_t-H_0\|_p^\alpha\|V\|_p^{m-1}$$
%   $$=\const\, \|f^{(m)}_p\|_{\Lambda_{\alpha}}\,t\,\|V\|_p^{\alpha+m-1}=\const\,
% \|f^{(m)}_p\|_{\Lambda_{\alpha}}\,t\,\|V\|_p^{p-1}= O\left( \left\| V
% \right\|_p^{p-1} \right), $$
%   where $\alpha=p-m.$ Now I am going to check the conditions of Theorem
% \ref{HolderInFull}. Since $m<p\le m+1,$ we have that $0<\alpha\le 1$ (It must be
% $0<\alpha<1$)
%   and $\sum_{j=0}^m\frac 1p=\frac{m+1}p\le 1$ (It must be $\sum_{j=0}^m\frac
%   1p<1$), and
%   $\frac{p-1}p=\frac{p-m}p+\frac{m}p=\frac{\alpha}p+\sum_{j=1}^m\frac 1p.$  }

\end{document}